\documentclass[11pt]{amsart}
\usepackage{amsmath,amssymb,amsfonts}
\usepackage[mathscr]{eucal}

\usepackage{hyperref}
\hypersetup{
    colorlinks=true,
    linkcolor=blue,
    citecolor=red
}
\usepackage{cleveref}
\usepackage{mathtools} 
\usepackage[margin=1.4in]{geometry}

\usepackage{enumerate}
\newenvironment{enumroman}{\begin{enumerate}[\upshape (i)]}
                                                {\end{enumerate}}
\usepackage[textwidth=3cm, textsize=small, colorinlistoftodos]{todonotes}

\usepackage{tikz}
\usetikzlibrary{decorations.pathreplacing,angles,quotes}
\usetikzlibrary{matrix, arrows}

\theoremstyle{plain}
\newtheorem{theorem}[equation]{Theorem}
\newtheorem{cor}[equation]{Corollary}
\newtheorem{prop}[equation]{Proposition}
\newtheorem{lemma}[equation]{Lemma}

\theoremstyle{definition}
\newtheorem{definition}[equation]{Definition}
\newtheorem{example}[equation]{Example}
\newtheorem{remark}[equation]{Remark}

\newtheorem{convention}[equation]{Convention}

\numberwithin{equation}{section}

\input xy
\xyoption{all}

\def \A {\mathcal {A}}
\def \B {\mathcal {B}}
\def \C {\mathcal {C}}
\def \D {\mathcal {D}}

\def \M {\mathcal {M}}
\def \P {\mathcal{P}}
\def \s {\mathcal{S}}
\def \V {\mathcal {V}}
\def \X {\mathcal{X}}

\newcommand{\colim}{\operatorname{colim}}

\newcommand{\Cyl}{\operatorname{Cyl}}
\newcommand{\Deltaop}{\Delta^{\op}}
\newcommand{\Fun}{\operatorname{Fun}}
\newcommand{\hf}{\operatorname{hf}}

\newcommand{\hocolim}{\operatorname{hocolim}}
\newcommand{\hofiber}{\operatorname{hofiber}}
\newcommand{\hocofiber}{\operatorname{hocofiber}}
\newcommand{\holim}{\operatorname{holim}}
\newcommand{\Hom}{\operatorname{Hom}}

\newcommand{\id}{\operatorname{id}}
\newcommand{\ifiber}{\operatorname{ifiber}}
\newcommand{\Map}{\operatorname{Map}}

\newcommand{\op}{\operatorname{op}}

\title{Cofibrantly generated model structures for functor calculus}

\author[Bandklayder]{Lauren Bandklayder}
\address{Boston Public Schools, 3200 Washington Street, Boston, MA, USA}
\email{lbandklayder@gmail.com}

\author[Bergner]{Julia E. Bergner}
\address{Department of Mathematics, University of Virginia, Charlottesville, VA, USA}
\email{jeb2md@virginia.edu}

\author[Griffiths]{Rhiannon Griffiths}
\address{Department of Mathematics, Cornell University, Ithaca, NY, USA}
\email{rhiannon.griffiths@cornell.edu}

\author[Johnson]{Brenda Johnson}
\address{Department of Mathematics, Union College, Schenectady, NY, USA}
\email{johnsonb@union.edu}

\author[Santhanam]{Rekha Santhanam}
\address{Department of Mathematics, IIT\ Bombay, Powai, Mumbai, India}
\email{reksan@iitb.ac.in}



\begin{document}

\begin{abstract}
    Model structures for many kinds of functor calculus can be obtained by applying a theorem of Bousfield to a suitable category of functors.  In this paper, we give a general criterion for when model categories obtained via this approach are cofibrantly generated.  Our examples recover the homotopy functor and $n$-excisive model structures of Biedermann and R\"ondigs, but also include a model structure for the discrete functor calculus of Bauer, Johnson, and McCarthy.
\end{abstract}

\maketitle

\section{Introduction}
Functor calculi have been used to produce significant results in a wide range of fields, beginning with applications of the homotopy functor calculus of Goodwillie \cite{goodwillie3} to algebraic $K$-theory \cite{mccarthy} and $v_n$-periodic homotopy theory \cite{am}, \cite{heuts}.  In another direction, the manifold calculus of Goodwillie and Weiss \cite{gw}, \cite{weissemb} and orthogonal calculus of Weiss \cite{weissorth} have been used to study embeddings and immersions of manifolds, as well as characteristic classes  \cite{rw}, and spaces of knots \cite{ltv}, \cite{sinha}.  
 
The essential data of a functor calculus is a means of approximating a functor $F$ with a tower of functors under $F$
\vspace{-2mm}
\begin{center}
\begin{tikzpicture}[node distance=2cm, auto]

\node (A) {$F$};
\node (E) [node distance=1.5cm, below of=A] {$P_nF$};
\node (D) [left of=E] {$...$};
\node (C) [left of=D] {$P_1F$};
\node (B) [left of=C] {$P_0F$};
\node (F) [right of=E] {$P_{n+1}F$};
\node (G) [right of=F] {$...$};

\draw[->] (C) to node {} (B);
\draw[->] (D) to node {} (C);
\draw[->] (E) to node {} (D);
\draw[->] (F) to node {} (E);
\draw[->] (G) to node {} (F);

\draw[->] (A) to node {} (B);
\draw[->] (A) to node {} (C);
\draw[->] (A) to node {} (D);
\draw[->] (A) to node {} (E);
\draw[->] (A) to node {} (F);
\draw[->] (A) to node {} (G);

\end{tikzpicture}
\end{center}
in a way analogous to the Taylor series for a function; indeed, we refer to them as \emph{Taylor towers}.  In particular, each $P_nF$ should be thought of as some kind of \emph{degree $n$ approximation} to $F$.  In fact, in a functor calculus, the natural transformations $p_nF \colon F\rightarrow P_nF$ induce weak equivalences $P_n(p_nF), \, p_n(P_nF) \colon P_nF\rightarrow P_n^2F$ which in turn can be used to show that $F\rightarrow P_nF$ is, in a categorical sense, the best degree $n$ approximation to $F$.

With a suitable model structure on a category of functors in place, the endofunctor $P_n$ that assigns to a functor its degree $n$ approximation can be used to build a new model structure in which $F\rightarrow G$ is a weak equivalence precisely when $P_nF\rightarrow P_nG$ is a weak equivalence in the original model structure. This approach has been employed by Barnes and Oman \cite{bo} in the case of the orthogonal calculus, and by Biedermann, Chorny, and R\"ondigs \cite{bcr}, \cite{br} and by Barnes and Eldred \cite{be1} in the case of the homotopy functor calculus. It has led to means by which different functor calculi can be extended to new contexts \cite{taggart} and compared to one another \cite{be2}, \cite{taggartcomp}, and provided the context in which to strengthen classifications of homogenous degree $n$ functors \cite{br}.  

One can also work in a more flexible setting than model categories, such as some well-behaved model for $(\infty,1)$-categories, which is the approach taken by Lurie \cite{lurie} and Pereira \cite{pereira}.  As their work demonstrates, one can develop a good general theory for functor calculus in this kind of setting as well, but we do not take that approach here, since we are particularly interested in the nuances of the model structures.  

The starting point for this paper was the desire to develop a similar approach for the discrete functor calculus of Bauer, McCarthy, and the fourth-named author \cite{bjm}.  Our hope is to use the model structure for discrete calculus developed here to recast existing comparisons between between the homotopy functor and discrete calculi, such as that found in \cite{bjm}, in model category-theoretic terms, and to develop a model category classification of homogeneous degree $n$ functors similar to the one for $n$-homogeneous functors in \cite{br}.   In addition, we have sought to define the model structure for discrete calculus in a manner that could be conveniently adapted for general classes of functor calculi. In particular, we would like for it to be applicable to those that can be defined using towers of comonads  in the same way as the discrete calculus does. A general approach to such calculi is being developed by the fourth author and Kathryn Hess.   To this end, we structure this paper around Theorems \ref{bousintro} and \ref{maintintro} below, and we show how to use them to build cofibrantly generated model structures corresponding to both the homotopy functor and discrete calculi.   

Model structures for abelian versions of the discrete functor calculus were established by Renaudin \cite{renaudin} and Richter \cite{richter}.   In general, the existence of such functor calculus model structures is guaranteed by a theorem of Bousfield and Friedlander \cite{bf}, which was modified by Bousfield for his work on telescopic homotopy theory \cite{bousfield}.   We use the following version of this theorem, derived from Bousfield's, in the present work; see Theorem \ref{BousfieldQ} and Corollary \ref{simplifiedA3}. The three versions differ slightly in the third axiom.  

\begin{theorem} \label{bousintro}  
Let $\M$ be a right proper model category together with an endofunctor $Q \colon \M \rightarrow \M$ and a natural transformation $\eta \colon \id \Rightarrow Q$ satisfying the following axioms:
\begin{itemize}
\item[(A1)] the endofunctor $Q$ preserves weak equivalences in $\M$;

\item[(A2)] the maps $\eta_{QX}, Q(\eta_X) \colon QX \rightarrow Q^2X$ are both weak equivalences in $\M$; and

\item[(A3')] $Q$ preserves homotopy pullback squares. 
\end{itemize}
Then there exists a right proper model structure, denoted by $\M_Q$, on the same underlying category as $\M$ with the same cofibrations as $\M$ and in which $X\rightarrow Y$ is a weak equivalence in $\M_Q$ when $QX\rightarrow QY$ is a weak equivalence in the original model structure on $\M$.  Furthermore, if $\M$ has the structure of a simplicial model category, then so does $\M_Q$.
\end{theorem}

A further version of this result by Stanculescu \cite{stanculescu} drops the right properness assumption, but since we have found this condition essential to our arguments here, we retain it.   We also note that the second axiom requires precisely the property described above that guarantees that $P_nF$ is, in a categorical sense, the best degree $n$ approximation to the functor $F$. In this sense, Bousfield and Friedlander's theorem and its variants seem tailor-made for constructing model structures for functor calculi. Localizations that are produced via these theorems are sometimes referred to as \emph{Bousfield-Friedlander localizations}, and the endofunctors used to produce them are examples of what are often referred to in the literature as \emph{Quillen idempotent monads}.  For specificity here we use the terminology \emph{Bousfield endofunctor} to emphasize that we are assuming the axioms of Theorem \ref{bousintro}.  We review Bousfield's version of the localization theorem and deduce Theorem \ref{bousintro} from it in Section \ref{s:QTheorem}, then apply it to obtain the $\hf$- and $n$-excisive model structures of \cite{br} in Sections \ref{fcexamples} and \ref{n-excisive model}, and the desired discrete degree $n$ model structure in Section \ref{Degreenmodel}.  

However, applying Theorem \ref{bousintro} does not immediately give the additional structure of a cofibrantly generated model category.  Facing the same challenge, Biedermann and R\"ondigs \cite{br} develop a simplicially enriched version of Goodwillie's homotopy functor calculus in such a way that their model structure for $n$-excisive functors is cofibrantly generated.  Because this additional structure is quite powerful, we want to employ a similar strategy to develop a degree $n$ model structure for discrete functor calculus.

Indeed, we develop a systematic way to show when model structures obtained from this theorem are cofibrantly generated, via criteria for when arguments of the kind used by Biedermann and R\"ondigs can be applied.  As a result, we are able to present not only the new example of discrete functor calculus, but also to show that the homotopy functor and $n$-excisive model structures arise  as consequences of a general theorem.  Our hope is that, together with the presentation of each of these three model structures, this general result will shed light on the shared features across these different model structures and the types of technical details that must be checked in specific examples.  With this goal in mind, even for the model structures that were already known, we emphasize some of the details in the arguments.  At the same time, the general theorem should enable us quickly to establish that the model structures arising from the functor calculi in the forthcoming work of the fourth-named author and Hess, which are constructed using methods similar to those used for the discrete calculus, are also cofibrantly generated.

In particular, one feature of Biedermann and R\"ondigs' cofibrant generation proofs is that they rely on the construction of equivalent formulations of functors such as $P_n$ in terms of representables.  Although our proofs do rely on their methods, we are able to circumvent the need for these replacements and use their original formulations instead.  In our approach, these representable functors appear in the form of what we call \emph{test morphisms} for a Bousfield endofunctor in the following result, which is restated more precisely as Theorem \ref{finalcofib} below.  

\begin{theorem} \label{maintintro} 
    Suppose that $\Fun(\C,\D)$  is a cofibrantly generated right proper model structure on a category of simplicial functors in which all fibrations are also levelwise fibrations, with some modest additional assumptions on $\C$ and $\D$.  If $Q$ is a Bousfield endofunctor of $\Fun(\C,\D)$ that admits a collection of test morphisms, then the model structure on $\Fun(\C,\D)_Q$ is cofibrantly generated.
\end{theorem}  

In practice, test morphisms represent maps in $\C$ that, for any functor $F$ in $\Fun(\C,\D)$, are converted to weak equivalences by $QF$.  After proving this theorem, we identify test morphisms for the $\hf$-, $n$-excisive, and degree $n$ model structures induced by Bousfield's theorem to prove that they are cofibrantly generated.  

\subsection{Outline of the paper}

In Section \ref{s:smcbackground} we review properties of  right proper  and simplicial model categories that will be used in subsequent sections, and  we look more specifically at simplicial model categories of functors, which are the examples of interest in this paper.  In Section \ref{s:fcbackground} we summarize both the homotopy functor calculus of Goodwillie and the discrete functor calculus of Bauer, Johnson, and McCarthy.  We present the localization techniques for model categories that we use in Section \ref{s:QTheorem}, and we apply them to get model structures for homotopy functors, $n$-excisive functors, and degree $n$ functors in Sections \ref{fcexamples}, \ref{n-excisive model}, and \ref{Degreenmodel}, respectively.  We present our main theorem for cofibrant generation in Section \ref{s:cofibrant} and include the homotopy functor model structure there as a guiding example.  We then apply the main theorem to the $n$-excisive model structure in Section \ref{s:nexccof} and to the degree $n$ model structure in Section \ref{s:degreencof}.

\subsection{Acknowledgements}

This work was done as part of the Women in Topology Workshop in August 2019, supported by the Hausdorff Research Institute for Mathematics, NSF grant DMS-1901795, the AWM ADVANCE grant NSF-HRD-1500481, and Foundation Compositio Mathematica.  JB was partially supported by NSF grants DMS-1659931 and DMS-1906281. BJ was partially supported by the Union College Faculty Research Fund.

We would like to thank the anonymous referee for several suggestions for the improvement of the paper.

\section{Background on simplicial model categories} \label{s:smcbackground} \label{s:simplicialfunctor}

In this section, we  review some key facts about model categories that we need, focusing on some features of right proper and simplicial model categories.  Of particular importance are some results about simplicial model categories that are difficult to find in the literature.

We begin with a number of properties of homotopy pullbacks in right proper model categories that we use throughout this paper, stated here for ease of reference. The first provides some standard means of constructing and comparing homotopy pullbacks.  

\begin{prop} \cite[13.3.4, 13.3.8]{hirschhorn} \label{homotopyinvariancepullback} \label{rightproperpullback}
Let $\M$ be a right proper model category, and let 
\vspace{-2mm}
\[ \begin{tikzpicture}[node distance=1.7cm, auto]

\node (A) {$X$};
\node (B) [right of=A] {$Z$};
\node (E) [right of=B] {$Y$};
\node (C) [node distance=1.4cm, below of=A] {$X'$};
\node (D) [right of=C] {$Z'$};
\node (F) [right of=D] {$Y'$};

\draw[->] (A) to node {} (B);
\draw[->] (E) to node {} (B);
\draw[->] (C) to node {} (D);
\draw[->] (F) to node {} (D);
\draw[->] (A) to node [swap] {$\simeq$} (C);
\draw[->] (B) to node {$\simeq$} (D);
\draw[->] (E) to node {$\simeq$} (F);

\end{tikzpicture} \] 
be a diagram in $\M$ where the vertical maps are all weak equivalences. The induced map between the homotopy pullback of the diagram $X \rightarrow Z \leftarrow Y$ and the homotopy pullback of the diagram $X' \rightarrow Z' \leftarrow Y'$ is a weak equivalence. Moreover, if, in any diagram $X\rightarrow Z\leftarrow Y$ in $\M$, at least one of the maps is a fibration, then the induced map from the ordinary pullback of the diagram to the homotopy pullback is a weak equivalence.
\end{prop}

The next result is a kind of ``two-out-of-three property" for homotopy pullbacks.

\begin{prop} \cite[13.3.15]{hirschhorn} \label{pullbackcomposite} 
Let $\M$ be a right proper model category.  If the right-hand square in the diagram
\vspace{-2mm}
\[ \begin{tikzpicture}[node distance=1.7cm, auto]

\node (A) {$X$};
\node (B) [right of=A] {$Y$};
\node (E) [right of=B] {$Z$};
\node (C) [node distance=1.4cm, below of=A] {$U$};
\node (D) [right of=C] {$V$};
\node (F) [right of=D] {$W$};

\draw[->] (A) to node {} (B);
\draw[->] (B) to node {} (E);
\draw[->] (C) to node {} (D);
\draw[->] (D) to node {} (F);
\draw[->] (A) to node {} (C);
\draw[->] (B) to node {} (D);
\draw[->] (E) to node {} (F);

\end{tikzpicture} \] 
is a homotopy pullback square, then the left-hand square is a homotopy pullback if and only if the composite square is a homotopy pullback.
\end{prop}

Finally, the following result gives a useful criterion for identifying homotopy pullback squares together with a nice property of homotopy pullback squares.  It is proved in \cite[3.3.11(1ab)]{mv} for the category of topological spaces, but the argument works for any right proper model category. 

\begin{prop} \label{properwe}
Consider a commutative square in a right proper model category $\M$: 
\begin{center}
\begin{tikzpicture}[node distance=1.7cm, scale=1]

\node (A) {$X$};
\node (B) [right of=A] {$Y$};
\node (C) [node distance=1.4cm, below of=A] {$U$};
\node (D) [right of=C] {$V$.};

\draw[->] (A) to node {} (B);
\draw[->] (A) to node [swap]{} (C);
\draw[->] (B) to node {} (D);
\draw[->] (C) to node [swap] {} (D);

\end{tikzpicture} 
\end{center}
\begin{enumerate}
    \item \label{properwea} 
    If the square is a homotopy pullback square and $Y\rightarrow V$ is a weak equivalence, then $X\rightarrow U$ is a weak equivalence.
    
    \item \label{properweb} 
    If both vertical arrows in the square are weak equivalences, then the square is a homotopy pullback square.
\end{enumerate}
\end{prop}

The model categories we consider in the paper are simplicially enriched model categories. We recall the definition here, with the axioms labelled according to the usual convention. A more detailed overview of simplicially enriched model categories can be found in the prequel to this paper \cite{proc}. 

\begin{definition}\label{SimpModelCat} \cite[9.1.6]{hirschhorn}
A \emph{simplicial model category} is a model category $\M$ that is also a simplicial category, i.e., a category enriched in the closed monoidal category $\s$ of simplicial sets, such that the following two axioms hold.
\begin{itemize}
    \item[(SM6)] The category $\M$ is both tensored and cotensored over $\s$. In particular, there are natural isomorphisms
    \[ \Map(X \otimes K, Y) \cong \Map(X,Y)^K \cong \Map(X, Y^K). \]
    for each pair of objects $X$ and $Y$ in $\M$ and simplicial set $K$.    
   
    \item[(SM7)] If $i \colon K \rightarrow L$ is a cofibration of simplicial sets and $p \colon X \rightarrow Y$ is a fibration in $\M$, then the induced morphism 
    \[ X^L \rightarrow X^K \times_{Y^K} Y^L \]
   is a fibration in $\D$ that is a weak equivalence if either $i$ or $p$ is.
\end{itemize}
\end{definition}
\begin{remark}\cite[9.3.7]{hirschhorn}\label{tensorcotensorrmk} Assuming axiom (SM6) holds, axiom (SM7)  is equivalent to 
\begin{itemize}
    \item [(SM7')]
If $i \colon K \rightarrow L$ is a cofibration of simplicial sets and $j \colon X \rightarrow Y$ is a cofibration in $\M$, then the induced morphism 
\[ X \otimes L \amalg_{X \otimes K} Y \otimes K \rightarrow Y \otimes L \]
is a cofibration in $\D$ that is a weak equivalence if either $i$ or $j$ is.
\end{itemize}
\end{remark}

\begin{example} \cite[9.1.13]{hirschhorn}
The model category $\s$ of simplicial sets  can be regarded as a simplicial model category.  Given any simplicial sets $X$ and $K$, we can take the tensor structure $X \otimes K$ to be the cartesian product $X \times K$ and the cotensor structure $X^K$ to be $\Map(K,X)$. 
\end{example}

We recall a result that is well-known in the model category of spaces, but holds in any simplicial model category; see, for example \cite[\S2.4]{dro}.

\begin{prop} \label{cylinder} 
Let $\M$ be a simplicial model category and let $f \colon A \to B$ be a morphism in $\M$ with $A$ cofibrant. Then there exists a factorization $f=g i_f$ of $f$ where $i_f$ is a cofibration and $g$ is a simplicial homotopy equivalence.
\end{prop}

Our main objects of study are categories whose objects are themselves functors between fixed categories.

\begin{convention} \label{convention}
From this point onward we assume that $\C$ is an essentially small simplicial category and that $\D$ is a cofibrantly generated right proper simplicial model category.  We denote by $\Fun(\C,\D)$ the category whose objects are simplicial functors $\C \rightarrow \D$ and whose morphisms are simplicial natural transformations.  Note that simplicial natural transformations are defined analogously to simplicial functors; see \cite[\S 1.2]{kelly} or \cite[2.10]{proc}.
\end{convention}

The following result tells us that there exists a model structure on the category of simplicial functors $\Fun(\C,\D)$ induced by the model structure on $\D$.  Note that we need some further technical assumptions because we consider simplicial functors between simplicial categories, compared to results such as \cite[11.6.1]{hirschhorn} for ordinary functors.  We omit some of these assumptions in the following theorem, but refer the reader to \cite[4.32]{gm} for the precise statement.

\begin{theorem} \cite[4.2, 4.3, 5.10]{proc}\label{t:Funmodel}
Assuming $\D$ satisfies some mild conditions on the set of generating acyclic cofibrations, the category $\Fun(\C,\D)$ has a cofibrantly generated right proper model structure, called the \emph{projective model structure}, in which a morphism $F \rightarrow G$ in $\Fun(\C,\D)$ is a weak equivalence or a fibration if it is one levelwise, i.e., $FA \rightarrow GA$ is a weak equivalence or fibration, respectively, in $\D$ for all objects $A$ of $\C$.

Furthermore, it has the structure of a simplicial model category. The tensor and cotensor structures are defined by $(F \otimes K)(A) = FA \otimes K$ and $(F^K)(A) = FA^K$, respectively, for each object $A$ in $\C$. 
\end{theorem}

Given that the category $\Fun(\C,\D)$ is enriched in $\s$, one might wonder whether it is also enriched in $\Fun(\C,\s)$. However, such an enrichment would require $\Fun(\C,\s)$ to be a closed monoidal category, which is not true in general.  Nevertheless, we show in Proposition \ref{FunAdjunctions} and Lemma \ref{pullbackgen} that $\Fun(\C,\D)$ does satisfy generalizations of axioms (SM6) and (SM7) from Definition \ref{SimpModelCat}, and therefore enjoys many of the properties of a model category enriched over $\Fun(\C,\s)$.  

\begin{definition} \label{representable}
For each object $C$ of $\C$, the simplicial functor $R^C \colon \C \rightarrow \s$ \emph{represented by} $C$ sends each object $A$ of $\C$ to the simplicial set $\Map_\C(C,A)$. Dually, the simplicial functor $R_C \colon \C^{\op} \rightarrow \s$ sends each object $A$ of $\C$ to  $\Map_\C(A,C)$.
\end{definition}

The following definition is critical to many of our arguments in this paper.

\begin{definition} \label{Fend}
Let $F \colon \C \rightarrow \D$ and $X \colon \C \rightarrow \s$ be simplicial functors.  The \emph{evaluated cotensor} of the pair $(F,X)$ is the equalizer  
\[ F^X := \int_{A} FA^{XA} \rightarrow \prod_{A} FA^{XA} \rightrightarrows \prod_{A,B} FB^{(XA^{\Map_{\C}(A,B)})} \] 
in $\D$ whose parallel morphisms are described explicitly in \cite[\S 3]{proc}.  Here, $A$ and $B$ range over objects of $\C$.
\end{definition}

\begin{remark} \label{cotensor} \label{brremark}  
\leavevmode
\begin{itemize}
\item As suggested by the notation $F^X = \int_{A} FA^{XA}$, the evaluated cotensor can be thought of as a generalization of an ordinary end.  In the special case when $\D = \s$, the assignment $F^{X}(A,B) = FB^{XA}$ defines a simplicial bifunctor $\C^{\op} \times \C \rightarrow \s$ whose end is precisely $F^X$.

\item This construction is originally due to Biedermann and R\"ondigs \cite[2.5]{br}, who use the notation $\mathbf{hom}(X,F)$.  We have chosen the name ``cotensor'' and the notation we use here to emphasize the fact that it behaves much like an ordinary cotensor; see Proposition \ref{FunAdjunctions}. 
\end{itemize}
\end{remark} 

Recall that a simplical adjunction satisfies the usual condition for an adjunction, but using mapping spaces rather than hom sets.  It can be thought of as a special case of a $\V$-adjunction for categories enriched in a general $\V$; see \cite[\S 1.11]{kelly}. The reference for the following proposition, which is an analogue of axiom (SM6) of Definition \ref{SimpModelCat} but for the evaluated cotensor, is stated at that level of generality. 

\begin{prop} \label{FunAdjunctions} \cite[4.4, 4.5, 4.6]{proc}  
For any simplicial functors $X \colon \C \rightarrow \s$ and $F \colon \C \rightarrow \D$ and object $D$ of $\D$ there are simplicial adjunctions
\vspace{-2mm}
\begin{center}
\begin{tikzpicture}[node distance=3.3cm, auto]

\node (A) {$\D$};
\node (B) [right of=A] {$\Fun(\C, \D)$};
\draw[->, bend left=35] (A) to node {$X \otimes -$} (B);
\draw[->, bend left=35] (B) to node {$(-)^X$} (A);
\node (W) [node distance=1.65cm, right of=A] {$\perp$};

\node (C) [node distance=1.5cm, right of=B] {$\D$};
\node (D) [right of=C] {$\Fun(\C, \s)^{\op}$};
\draw[->, bend left=35] (C) to node {$\Map_\D(-,F)$} (D);
\draw[->, bend left=35] (D) to node {$F^{(-)}$} (C);
\node (U) [node distance=1.65cm, right of=C] {$\perp$};

\node (E) [node distance=2.2cm, right of=D] {$\Fun(\C, \s)$};
\node (F) [right of=E] {$\Fun(\C, \D)$};
\draw[->, bend left=35] (E) to node {$- \otimes D$} (F);
\draw[->, bend left=35] (F) to node {$\Map_\D(D,-)$} (E);
\node (V) [node distance=1.65cm, right of=E] {$\perp$};

\end{tikzpicture} 
\end{center}
where $(X \otimes D)(C) := XC \otimes D$ and $\Map_{\D}(D,F)(A) := \Map_{\D}(D,FA)$. In particular, there are simplicial natural isomorphisms
\[ \Map_{\Fun(\C, \D)} (X \otimes D, F)  \cong  \Map_{\Fun(\C, \s)} ( X, \Map_{\D}(D, F)) \cong \Map_{\D}(D, F^X). \]
\end{prop}

The evaluated cotensor also provides an enriched version of the (co)Yoneda lemma, stated below. The proof follows an argument similar to the one outlined in \cite[\S 2.4]{kelly} for the case where $\C$ is a $\V$-category for some closed monoidal category $\V$ and $\D=\V$.

\begin{lemma} \cite[3.5, 3.9]{proc} \label{Yoneda} 
Let $F \colon \C \rightarrow \D$ be a simplicial functor.  For each object $C$ of $\C$ there is a natural isomorphism 
\[ FC \cong F^{R^C} = \int_{A} FA^{R^C(A)} = \int_{A} FA^{\Map_{\C}(C,A)}. \]
Dually, we have
\vspace{-2mm}
\[ FC \cong F^{R_C} = \int^{A} FA^{R_C(A)} = \int^{A} FA \otimes {\Map_{\C}(C,A)}. \]
\end{lemma}

The next lemma establishes that a version of the condition (SM7) in Definition \ref{SimpModelCat} holds for the evaluated cotensor.  

\begin{lemma}\label{pullbackgen}
Consider $\Fun(\C,\D)$ and $\Fun(\C,\s)$, each equipped with the projective model structure. If $X \rightarrow Y$ is a cofibration in $\Fun(\C, \s)$ and $F \rightarrow G$ is a fibration in $\Fun(\C,\D)$, then $F^Y \rightarrow F^X \times_{G^X} G^Y$ is a fibration in $\D$ that is a weak equivalence if either $X \rightarrow Y$ or $F \rightarrow G$ is.
\end{lemma} 

\begin{proof}
Let $F \rightarrow G$ be a fibration in $\Fun(\C,\D)$, so   
that $FA \rightarrow GA$ is a fibration in $\D$ for any object $A$ of $\C$. If $K \rightarrow L$ is a cofibration in $\s$, then by (SM7) the pullback-corner map for the ordinary cotensor
\vspace{-2mm}
\[ FA^L \rightarrow FA^K \times_{GA^K} GA^L \]
is a fibration in $\D$.  Thus, if $D \rightarrow E$ is an acyclic cofibration in $\D$, there exists a lift in any diagram of the form
\vspace{-2mm}
\begin{center}
\begin{tikzpicture}[node distance=1.7cm, auto]

\node (A) {$D$};
\node (B) [below of=A]{$E$};
\node (C) [node distance=3.5cm, right of=A] {$FA^L$};
\node (D) [node distance=3.5cm, right of=B] {$FA^K \times_{GA^K} GA^L$.};

\draw[left hook->] (A) to node {$\simeq$} (B);
\draw[left hook->] (A) to node [swap] {$i$} (B);
\draw[->] (A) to node {} (C);
\draw[->] (B) to node [swap] {} (D);
\draw[->>] (C) to node {} (D);
\draw[->, dashed] (B) to node {} (C);

\end{tikzpicture}
\end{center}
Using the natural isomorphism $\Map_{\D}(E,FA)^L \cong \Map_{\D}(E, FA^L)$ given by taking $\M = \D$ in axiom (SM6),  we have a commutative diagram
\vspace{-2mm}
\begin{center}
\begin{tikzpicture}[node distance=1.7cm, auto]

\node (A) {$K$};
\node (B) [below of=A]{$L$};
\node (C) [node distance=5.7cm, right of=A] {$\Map_{\D}(E,FA)$};
\node (D) [node distance=5.7cm, right of=B] {$\Map_{\D}(D,FA) \times_{\Map_{\D}(D,GA)} \Map_{\D}(E,GA)$};

\draw[left hook->] (A) to node [swap] {} (B);
\draw[->] (A) to node {} (C);
\draw[->] (B) to node [swap] {} (D);
\draw[->] (C) to node {} (D);
\draw[->, dashed] (B) to node {} (C);

\end{tikzpicture}
\end{center}
in $\s$. Since $K \rightarrow L$ is assumed to be a cofibration, and using the fact that $A$ is an arbitrary object of $\C$, we can conclude that the map
\[ \Map_\D(E,F) \rightarrow \Map_\D(D,F) \times_{\Map_\D(D,G)} \Map_\D(E,G) \]
is an acyclic fibration in $\Fun(\C,\s)$. It follows that if $X \rightarrow Y$ is a cofibration in $\Fun(\C,\s)$, then a lift exists in any diagram of the form
\vspace{-2mm}
\begin{center}
\begin{tikzpicture}[node distance=1.7cm, auto]

\node (A) {$X$};
\node (B) [below of=A]{$Y$};
\node (C) [node distance=5.3cm, right of=A] {$\Map_{\D}(E,F)$};
\node (D) [node distance=5.3cm, right of=B] {$\Map_{\D}(D,F) \times_{\Map_{\D}(D,G)} \Map_{\D}(E,G)$.};

\draw[left hook->] (A) to node [swap] {} (B);
\draw[->] (A) to node {} (C);
\draw[->] (B) to node [swap] {} (D);
\draw[->>] (C) to node {$\simeq$} (D);
\draw[->, dashed] (B) to node {} (C);

\end{tikzpicture}
\end{center}
Applying axiom (SM6) once more, we equivalently obtain a lift 
\vspace{-2mm}
\begin{center}
\begin{tikzpicture}[node distance=1.7cm, auto]

\node (A) {$D$};
\node (B) [below of=A]{$E$};
\node (C) [node distance=3cm, right of=A] {$F^Y$};
\node (D) [node distance=3cm, right of=B] {$F^X \times_{G^X} G^Y$,};

\draw[left hook->] (A) to node {$\simeq$} (B);
\draw[left hook->] (A) to node [swap] {$i$} (B);
\draw[->] (A) to node {} (C);
\draw[->] (B) to node [swap] {} (D);
\draw[->] (C) to node {} (D);
\draw[->, dashed] (B) to node {} (C);

\end{tikzpicture}
\end{center}
so we have thus shown that the map 
\[ F^Y \rightarrow F^X\times_{G^X} G^Y \]
is a fibration in $\D$. The proof that $F^Y \rightarrow F^X\times_{G^X} G^Y$ is also a weak equivalence if either $X \rightarrow Y$ or $F \rightarrow G$ is a weak equivalence is analogous.
\end{proof}

\begin{cor} \label{cotensorwefib}
\leavevmode
\begin{enumerate}
    \item If $X \rightarrow Y$ is a cofibration in $\Fun(\C,\s)$ and $F$ is a fibrant object in $\Fun(\C,\D)$, then $F^Y \rightarrow F^X$ is a fibration in $\D$ that is a weak equivalence if $X \rightarrow Y$ is.
    
    \item If $X$ is a cofibrant object in $\Fun(\C,\s)$ and $F \rightarrow G$ is a fibration in $\Fun(\C,\D)$, then $F^X \rightarrow G^X$ is a fibration in $\D$ that is a weak equivalence if $F \rightarrow G$ is.
\end{enumerate}
\end{cor}

\begin{remark}
Lemma \ref{pullbackgen} and Corollary \ref{cotensorwefib} also hold when the model structure on $\Fun(\C,\D)$ is a localization of the projective model structure with the same cofibrations, and whose fibrations form a subclass of the projective fibrations.  In particular, they apply to model structures obtained by applying the localization given in Theorem \ref{BousfieldQ}, from which we get our main examples in this paper, as well as any left Bousfield localization in the sense of \cite{hirschhorn}. 
\end{remark}

Finally, we include a result about the interaction between the evaluated cotensor and homotopy (co)limits.

\begin{prop} \cite[6.5]{proc} \label{evcotensorholim}
    Let $\mathcal I$ be a small category and $\mathbf X$ an $\mathcal I$-diagram in $\Fun(\C,\s)$.  Let $\D$ be a cofibrantly generated simplicial model category such that the projective model structure on $\Fun(\C,\D)$ exists.  Then for any simplicial functor $F \colon \C \rightarrow D$, there is a natural isomorphism
     \vspace{-2mm}
    \[ F^{\underset{i}{\hocolim} \mathbf X_i} \cong \underset{i}{\holim} \, F^{\mathbf X_i}. \]
\end{prop}

\section{Background on functor calculus} \label{s:fcbackground}

In this section, we give a brief review of two different flavors of functor calculus: Goodwillie's homotopy functor calculus \cite{goodwillie3} and the discrete functor calculus of Bauer, Johnson, and McCarthy \cite{bjm}.  Each of these functor calculi associates to a functor $F$ a ``Taylor" tower of functors under $F$
\vspace{-2mm}
\begin{center}
\begin{tikzpicture}[node distance=2cm, auto]

\node (A) {$F$};
\node (E) [node distance=1.5cm, below of=A] {$P_nF$};
\node (D) [left of=E] {$...$};
\node (C) [left of=D] {$P_1F$};
\node (B) [left of=C] {$P_0F$};
\node (F) [right of=E] {$P_{n+1}F$};
\node (G) [right of=F] {$...$};

\draw[->] (C) to node {} (B);
\draw[->] (D) to node {} (C);
\draw[->] (E) to node {} (D);
\draw[->] (F) to node {} (E);
\draw[->] (G) to node {} (F);

\draw[->] (A) to node {} (B);
\draw[->] (A) to node {} (C);
\draw[->] (A) to node {} (D);
\draw[->] (A) to node {} (E);
\draw[->] (A) to node {} (F);
\draw[->] (A) to node {} (G);

\end{tikzpicture}
\end{center}
that are analogous to Taylor polynomial approximations to a function; the $n$th term $P_nF$ is ``degree $n$" in some sense and serves as a universal degree $n$ approximation to $F$.  We review the construction of the $n$th term in the towers for Goodwillie's calculus and the discrete calculus, and establish some properties that will be needed in subsequent sections.

We use the following notation.  For any $n\geq 1$, let $\P(\mathbf{n})$ denote the poset of subsets of the set $\mathbf{n}=\{1, \dots, n\}$.  Let $\P_0(\mathbf{n})$ denote the poset of non-empty subsets of $\mathbf{n}$ and $\P_{\leq 1}(\mathbf {n})$ the full subcategory of $\P(\mathbf{n})$ spanned by the subsets of cardinality less than or equal to $1$. An \emph{$n$-cubical diagram} in a category $\C$ is a functor $\X \colon \P(\mathbf{n}) \rightarrow \C$.  For more details about $n$-cubical homotopy theory see \cite{goodwillie2}, \cite[\S 4.1]{kuhn1}, or \cite{mv}.

\subsection{Goodwillie calculus}\label{s:Goodwillie calculus}

We  describe the first version of functor calculus that we consider, namely, the  calculus of homotopy functors as defined by Goodwillie. This calculus grew out of work on pseudoisotopy theory and Waldhausen's algebraic $K$-theory of spaces, and has led to important results in $K$-theory and homotopy theory; see, for example, \cite{am}, \cite{behrens}, \cite{dundasmccarthy}, \cite{kuhn2}, and \cite{mccarthy}.  The Taylor tower for this calculus is constructed in \cite{goodwillie3}, which serves as the reference for the definitions in this subsection. 

We note that instead of restricting our source and target categories to spaces and spectra as Goodwillie did, we work in the more general setting of simplicial model categories, an approach that Kuhn developed in \cite{kuhn1}.  In comparison to their approaches, we impose more structure on our category of functors and define $\Fun(\C,\D)$ to be the category whose objects are simplicial functors, and use a simplicial model structure on it to define an $n$-excisive model structure in Section \ref{n-excisive model}.   To start, we describe the construction of the Taylor tower as developed by Goodwillie and Kuhn and then, at the end of this section, we describe the adjustments that need to be made for our purposes. 

\begin{convention} \label{GoodwilleConventions} 
In addition to the assumptions on $\C$ and $\D$ from Convention \ref{convention}, we assume in this subsection that $\C$ is a simplicial subcategory of a simplicial model category that is closed under finite homotopy colimits and has a terminal object $\ast_{\C}$.
\end{convention}

For an object $A$ in $\C$ and a subset $U$ of $\mathbf{n}$, the \emph{fiberwise join} $A\ast U$ is the homotopy colimit in $\C$ of the $\P_{\leq 1}(U)$-diagram that assigns to $\varnothing$ the object $A$ and assigns to each one-element set $\{i\}$ the terminal object $\ast_\C$. If $U$ is a one-element set, then $A\ast U$ is the simplicial cone on $A$, and in general, $A\ast U$ consists of $|U|$ copies of the simplicial cone on $A$ glued along $A$. For a fixed object $A$ in $\C$, $A\ast -$ defines a functor from $\P(\mathbf{n})$ to $\C$. This functor plays a key role in the definition of the functors $P_nF \colon \C \rightarrow \D$ appearing in Goodwillie's Taylor tower. 

\begin{definition} \label{Pndefn}
For any functor $F \colon \C \rightarrow \D$, define the functors $T_nF$ and $P_nF$ by 
\[ T_nF (A)= \underset{U \in \P_0(\mathbf{n+1})}{\holim} F (A \ast U) \]
and 
\[ P_nF(A) = \underset{k}{\hocolim} \ T_n^kF(A), \] 
where the homotopy colimit is over a sequential diagram obtained by repeatedly applying the transformation $F \xrightarrow{t_nF} T_nF$ induced by the canonical map $A=A\ast \varnothing \rightarrow A \ast U$ for each $U$. These maps also yield a natural transformation $p_nF \colon F\rightarrow P_nF$.
\end{definition} 

\begin{remark} \label{Pnproperties}
The functor $P_nF$ is the $n$-\emph{excisive approximation} of $F$.  More precisely, it is an $n$-\emph{excisive functor}, that is, it takes strongly homotopy cocartesian $(n+1)$-cubical diagrams in $\C$ to homotopy cartesian $(n+1)$-cubes in $\D$, and the induced natural transformation $p_nF \colon F\rightarrow P_nF$ is appropriately universal with respect to this property \cite[1.8]{goodwillie3}.  
\end{remark}

\subsection{Discrete functor calculus} \label{s:discretefc}

Abelian functor calculus is a functor calculus developed for algebraic settings that builds Taylor towers using certain classical constructions, such as the stable derived functors defined by Dold and Puppe \cite{dp} and stable homology of modules over a ring $R$ with coefficients in a ring $S$ as defined by Eilenberg and Mac Lane \cite{em1}, \cite{em2}, as the first  terms in its Taylor towers. It has been used to study rational algebraic $K$-theory, Hochschild homology, and other algebraic constructions, and to make connections between functor calculus and the tangent and cartesian differential categories of Blute, Cockett, Cruttwell, and Seely \cite{bjort}, \cite{jm1}, \cite{jm2}, \cite{km}.  Discrete functor calculus arose as an adaptation of abelian functor calculus for functors of simplicial model categories. In general, the $n$th term in the discrete Taylor tower satisfies a weaker degree $n$ condition than the corresponding term in the Goodwillie calculus, but the towers agree for functors that commute with realization \cite[4.11, 5.4]{bjm}.  

For this kind of functor calculus, we need the notion of a stable model category.  First, we recall that a model category is \emph{pointed} if it has a zero object, that is, if the initial and terminal objects coincide.  

\begin{definition} \label{stablemc} 
A model category $\D$ is \emph{stable} if it is pointed and its homotopy category is triangulated.
\end{definition}

We do not go into the details of triangulated categories here, but refer the reader to \cite[Ch.\ 7]{hovey} for an overview aimed at understanding stable model categories.  For our purposes, the most important feature of stable model categories is that homotopy pullback squares are the same as homotopy pushout squares.

The stability condition is used to guarantee the $n$th term in the discrete calculus tower for a functor is really a degree $n$ functor, a result that is straightforward using the agreement of homotopy pushout and pullback diagrams in stable model categories; see \cite[4.5, 4.6, 5.4]{bjm}.

\begin{convention} \label{conventiondiscrete} 
In addition to the assumptions on $\C$ and $\D$ from Convention \ref{convention}, we assume in this subsection that $\mathcal C$ has finite coproducts and a terminal object $\ast_{\C}$, and that $\D$ is a stable model category in the sense of Definition \ref{stablemc}.  In particular, $\D$ has a zero object that we denote by $\star_{\D}$. 
\end{convention}

As we did in the previous subsection, we  describe the construction of degree $n$ approximations in the discrete calculus as it was done originally in \cite{bjm}, and then, in the next subsection, explain what is needed to ensure that the terms in the discrete Taylor tower are simplicial functors.  

To define the $n$th term in the discrete Taylor tower, we use a comonad constructed from the iterated homotopy fiber of a particular $(n+1)$-cubical diagram.

\begin{definition}  
Let $\mathcal X \colon \P(\mathbf 2) \rightarrow \D$ be a 2-cubical diagram.  Its \emph{iterated homotopy fiber} is
\vspace{-2mm}
\[ \ifiber(\mathcal X) := \hofiber \left( \hofiber \left( \vcenter{\xymatrix{X_\varnothing \ar[d] \\ X_2}} \right) \rightarrow \hofiber \left( \vcenter{\xymatrix{X_1 \ar[d] \\ X_{12}}} \right) \right). \]
\end{definition}

In other words, we take the homotopy fibers of the two vertical maps, which produces the horizontal map between  those homotopy fibers; the \emph{iterated} fiber is just the homotopy fiber of this induced map.  The iterated fiber of an $n$-cubical diagram $\X \colon \P(\mathbf n) \rightarrow \D$ can be defined analogously.  More details can be found in \cite[3.1]{bjm} and \cite[\S 3.4, \S5.5]{mv}. 

For an object $A$ in $\C$ and a functor $F \colon \C\rightarrow \D$, let $F^{\mathbf{n}}(A,-)$ be the $\P(\mathbf{n})$-diagram that assigns to $U \subseteq \mathbf{n}$ the object
\[ F^{\mathbf{n}}(A,U):=F \left(\coprod _{i\in \mathbf{n}} A_i(U) \right), \]
where
\vspace{-2mm}
\begin{equation} \label{e:Ai}
A_i(U):= \begin{cases}
A & i\notin U, \\
\ast_\mathcal{C} & i\in U.
\end{cases}
\end{equation} 
Furthermore, given a functor $F \colon \C \rightarrow \D$, we define the functor $F^{\P_0(\mathbf{2})}_n \colon \C \times {\P_0(\mathbf{2})}^{\times n} \rightarrow \D$ by 
\vspace{-2mm}
\[ F^{\P_0(\mathbf{2})}_n \left(A;(S_1, \dots, S_n)\right):=
\begin{cases}
F^{\mathbf{n}}(A,\varphi(S_1, \dots, S_n)) & S_i\neq \{2\} \text{ for all } i, \\
\star_{\D} & \text{otherwise,}
\end{cases} \]
where 
\vspace{-2mm}
\begin{equation} \label{e:phi}
\varphi(S_1, \dots, S_n)=\{i\ |\ S_i=\{1,2\}\}.
\end{equation}

\begin{example}
The diagram $F_2^{\P_0(\mathbf{2})}(A,-)$ is given by the diagram 
\vspace{-2mm}
\begin{center}
\begin{tikzpicture}[node distance=3cm, auto, scale=1]

\node (A) {$F(A \amalg A)$};
\node (B) [right of=A] {$F(\ast_\C \amalg A)$};
\node (X) [right of=B] {$\star_{\D}$};
\node (C) [node distance=1.5cm, below of=A] {$F(A \amalg \ast_\C)$};
\node (D) [node distance=1.5cm, below of=B] {$F(\ast_\C \amalg \ast_\C)$};
\node (Y) [node distance=1.5cm, below of=X] {$\star_{\D}$};
\node (E) [node distance=1.5cm, below of=C] {$\star_{\D}$};
\node (F) [node distance=1.5cm, below of=D] {$\star_{\D}$};
\node (Z) [node distance=1.5cm, below of=Y] {$\star_{\D}$.};

\draw[->] (A) to node {} (B);
\draw[->] (A) to node {} (C);
\draw[->] (C) to node {} (D);
\draw[->] (B) to node {} (D);

\draw[->] (E) to node {} (C);
\draw[->] (F) to node {} (D);
\draw[->] (X) to node {} (B);
\draw[->] (Y) to node {} (D);

\draw[->] (E) to node {} (F);
\draw[->] (Z) to node {} (F);
\draw[->] (X) to node {} (Y);
\draw[->] (Z) to node {} (Y);
\end{tikzpicture}
\end{center}
\end{example}

\begin{definition} \label{perp} 
For any functor $F \colon \C \rightarrow \D$ and any object $A$ of $\C$, define
\[ \bot_{n} F(A):=\underset{(S_1, \dots, S_n)}{\holim} \, F^{\P_0(\mathbf{2})}_n \left(A; (S_1, \dots, S_n) \right), \] 
where the homotopy limit is taken over the category ${\P_0(\mathbf{2})}^{\times n}$.
\end{definition}

Observe that $\bot_2F(A)$ is the iterated homotopy fiber of the diagram
\begin{center}
\begin{tikzpicture}[node distance=3cm, auto, scale=1]

\node (A) {$F(A \amalg A)$};
\node (B) [right of=A] {$F(\ast_\C \amalg A)$};
\node (C) [node distance=1.5cm, below of=A] {$F(A \amalg \ast_\C)$};
\node (D) [node distance=1.5cm, below of=B] {$F(\ast_\C \amalg \ast_\C)$.};

\draw[->] (A) to node {} (B);
\draw[->] (A) to node {} (C);
\draw[->] (C) to node {} (D);
\draw[->] (B) to node {} (D);

\end{tikzpicture}
\end{center}
More generally, $\bot_nF(A)$ is equivalent to the iterated homotopy fiber of the $\P(\mathbf{n})$-diagram that assigns to the set $U\subseteq \mathbf{n}$ the object $F^{\mathbf n}(A,U)$; see \cite[\S 3.1]{bjm} for details. 

\begin{definition} \label{degreendefn}
A functor $F \colon \C\rightarrow \D$ is \emph{degree $n$} provided that for any $(n+1)$-tuple of objects $\mathbf{X}=(X_1, \dots, X_{n+1})$ in $\C$, the iterated homotopy fiber of the diagram 
\begin{center}
\begin{tikzpicture}[node distance=4cm, auto, scale=1]

\node (A) {${\mathbf{n+1}}\supseteq U$};
\node (B) [right of=A] {$F \left( \coprod_{i\in \mathbf{n+1}}\mathbf{X}_i(U) \right)$};

\draw[|->] (A) to node {} (B);

\end{tikzpicture}
\end{center}
where 
\[ \mathbf{X}_i(U)=\begin{cases}X_i&i\notin U,\\
\ast_\C &i\in U,
\end{cases} \]
is weakly equivalent to the terminal object in $\D$.  This iterated homotopy fiber is called the \emph{$(n+1)$st cross effect of $F$ at $\mathbf{X}$} and is denoted by $cr_{n+1}F(\mathbf{X})$. Note that when $\mathbf{X}=(X, \dots, X)$ for an object $X$ in $\C$, $cr_{n+1}F(\mathbf X)=\bot_{n+1}F(X)$.
\end{definition}

To define a  degree $n$ approximation, we note that $\bot_n$ can be given the additional structure of a comonad.  Recall that a \emph{comonad} $(\bot, \Delta, \varepsilon)$ acting on a category $\A$ consists of an endofunctor $\bot \colon \A \rightarrow \A$ together with natural transformations $\Delta \colon \bot\rightarrow \bot\bot$ and $\varepsilon \colon \bot\rightarrow \id_\A$ satisfying certain identities.  For an object $A$ in $\A$, there is an associated simplicial object $\bot^* A$ given by
\vspace{-2mm}
\begin{center}
\begin{tikzpicture}[node distance=2.5cm, auto, scale=1]

\node (A) {$[k]$};
\node (B) [transform canvas={yshift=0.2ex}, right of=A] {$\bot^{k+1}A$};

\draw[|->] (A) to node {} (B);

\end{tikzpicture}
\end{center}
whose face and degeneracy maps are defined using the natural transformations $\varepsilon$ and $\Delta$.  See \cite[\S 8.6]{weibel} for more details, noting that the author uses the term ``cotriple" for what we are calling a ``comonad" here.  

The indexing category $\P_0(\mathbf{2})^{\times n}$ used for the homotopy limit that defines the functor $\bot_n$ induces both a comultiplication $\Delta_n$ and a counit $\varepsilon_n$. Together with the natural transformations $\Delta_n$ and $\varepsilon_n$, $\perp_n$ defines a  comonad.  Proving that $\perp_n$ is a comonad requires proving that certain diagrams are strictly commutative.  Since $\perp_n$ is defined using a homotopy limit and different models for homotopy limits may only agree up to weak equivalence, one needs to fix a choice of model for homotopy limits.  For the proof that $\perp_n$ is a comonad in \cite[\S 3]{bjm}, one also needs to ensure that this model satisfies certain standard properties, which are listed in Appendix A of \cite{bjm},  up to isomorphism, rather than weak equivalence.  As shown in \cite[App.\ A]{bjm}, the standard model for homotopy limits  in \cite[18.1.8]{hirschhorn} satisfies these conditions.  

Once one has established the existence of a comonad structure for $\perp_n$ with this choice of model for the homotopy limit, and shown that the functors $\Gamma_nF$ in Definition \ref{degreen} are in fact degree $n$, one can use a different model for homotopy colimit to construct $\Gamma_nF$, provided it is weakly equivalent to the original. Thus we can make the following definition.  

\begin{definition} \label{degreen}
For a functor $F \colon \C \rightarrow \D$, the \emph{degree $n$ approximation} of $F$ at an object $A$ is given by 
\[ \Gamma_n F(A):= \operatorname{hocofiber}\left(|\perp_{n+1}^{*+1} F(A)| \rightarrow F(A)\right), \] 
where $\perp_{n+1} F(A)$ is as in Definition \ref{perp}, and $\vert \bot_{n+1}^{\ast+1}F(A)\vert$ is the homotopy colimit over $\Deltaop$ of the standard simplicial object associated to the comonad $\bot_{n+1}$ acting on $F$. 
\end{definition}

\begin{remark} 
The functor $\Gamma_n F$ is degree $n$ in the sense of Definition \ref{degreendefn}, which is a weaker condition than the one for $n$-excisive functors, satisfied by the $n$th term in Goodwillie's Taylor tower.  However, the proof that $\Gamma_nF$ is degree $n$ uses a standard extra degeneracy argument for simplicial objects, whereas Goodwillie describes his proof that $P_nF$ is $n$-excisive as ``somewhat opaque" \cite[1.10]{goodwillie3}. As is the case with the Goodwillie tower, one can show that there is a natural transformation $F\rightarrow \Gamma_nF$ that is universal (up to homotopy) among degree $n$ functors with natural transformations from $F$ \cite[\S 5]{bjm}.
\end{remark}

\subsection{Polynomial approximations for simplicial functors}

To establish the existence of the $n$-excisive and degree $n$ model structures in Sections \ref{n-excisive model} and \ref{Degreenmodel}, we need to know that $P_nF$ and $\Gamma_nF$ are simplicial functors when $F$ is, that is, that $P_nF$ and $\Gamma_nF$ are objects in $\Fun(\C,\D)$ as we have defined it. We also need to know that $P_n$ and $\Gamma_n$ preserve weak equivalences in $\Fun(\C,\D)$.  These conditions require knowing that the models we use for the homotopy limit and homotopy colimit in $\D$ preserve weak equivalences and are suitably simplicial. 

\begin{remark} \label{c:holimPn}  
Here, we use the models for homotopy limits and homotopy colimits for simplicial model categories as described in \cite[18.1.2, 18.1.8]{hirschhorn}. To guarantee that our homotopy limits and colimits preserve levelwise weak equivalences, which need not be the case for these models \cite[Ch.\ 18]{hirschhorn}, we use simplicial fibrant and cofibrant replacement functors.   Such replacements exist for any cofibrantly generated simplicial model category by \cite[24.2]{shulman}, so it is safe to assume that there are models for homotopy limits and homotopy colimits in $\D$ that are simplicial and preserve levelwise weak equivalences of diagrams.
\end{remark}

We conclude the following.

\begin{prop}
For a simplicial functor $F$ in $\Fun(\C,\D)$, both $P_nF$  and $\Gamma_nF$ are simplicial functors.
\end{prop}

\begin{proof}
As defined, the functor $P_nF$ is a composite of fiberwise joins, the original functor $F$, homotopy limits, and a sequential homotopy colimit. Each of these components is simplicial.  In the case of the fiberwise joins, we can use \cite[7.15]{proc}.  The other cases are consequences of the assumptions we have made.  A similar argument gives us the result for $\Gamma_nF$. 
\end{proof}

We conclude this section by noting that using models for homotopy (co)limits that involve precomposition with (co)fibrant replacement functors affect the natural transformations from $F$ to $P_nF$ and $\Gamma_nF$; for example, we may only have  natural transformations from the cofibrant replacement of $F$ to $P_nF$ and $\Gamma_nF$.  We describe how to circumvent such issues in Sections \ref{n-excisive model} and \ref{Degreenmodel}.

\section{Bousfield's $Q$-Theorem} \label{s:QTheorem}  

In this section, we recall the approach to localization that we use in this paper, due to Bousfield and Friendlander.  Because the basic input is a model category equipped with an endofunctor $Q$, and we use Bousfield's later variant, we  refer to this result simply as ``Bousfield's $Q$-Theorem".   As is typical for a localization, we would like to produce a new model structure with more weak equivalences; we begin by defining the weak equivalences that we obtain from the endofunctor $Q$.

\begin{definition}
Let $\M$ be a model category, together with an endofunctor $Q \colon \M \rightarrow \M$.  A morphism $f$ in $\M$ is a \emph{$Q$-equivalence} if it is a weak equivalence after applying the functor $Q$; that is, if $Q(f)$ is a weak equivalence in $\M$.  
\end{definition}

We now state Bousfield's $Q$-Theorem.

\begin{theorem} \cite[9.3, 9.5, 9.7]{bousfield} \label{BousfieldQ} 
Let $\M$ be a right proper model category together with an endofunctor $Q \colon \M \rightarrow \M$ and a natural transformation $\eta \colon \id \Rightarrow Q$ satisfying the following axioms:
\begin{itemize}
\item[(A1)] the endofunctor $Q$ preserves weak equivalences in $\M$;

\item[(A2)] the maps $\eta_{QX}, Q(\eta_X) \colon QX \rightarrow Q^2X$ are both weak equivalences in $\M$; and

\item[(A3)] given a pullback square 
\vspace{-2mm}
\begin{center}
\begin{tikzpicture}[node distance=1.8cm, auto, scale=0.5]

\node (A) {$V$};
\node (B) [right of=A] {$X$};
\node (C) [node distance=1.5cm, below of=A] {$W$};
\node (D) [node distance=1.5cm, below of=B] {$Y$};

\draw[->] (A) to node {$f$} (B);
\draw[->] (A) to node {} (C);
\draw[->] (C) to node [swap] {$g$} (D);
\draw[->] (B) to node {$h$} (D);

\end{tikzpicture}
\end{center}
if $X$ and $Y$ are fibrant, $h \colon X \rightarrow Y$ is a fibration, $\eta_X \colon X \rightarrow QX$ and $\eta_Y \colon Y \rightarrow QY$ are weak equivalences, and $g \colon W \rightarrow Y$ is a $Q$-equivalence, then $f \colon V \rightarrow X$ is a $Q$-equivalence.
\end{itemize}
Then there exists a right proper model structure, denoted by $\M_Q$, on the same underlying category as $\M$ with weak equivalences the $Q$-equivalences and the same cofibrations as $\M$.  Furthermore, if $\M$ has the structure of a simplicial model category, then so does $\M_Q$. 
\end{theorem}  

The next two results provide  modifications of the hypotheses of this theorem.  The first shows that it suffices to check axiom (A2) on fibrant objects, while the second gives us a way to show that (A3) is satisfied.

\begin{cor} \label{fibrantA2} 
Let $\M$ be a right proper model category together with an endofunctor $Q \colon \M \rightarrow \M$ and natural transformation $\eta \colon \id \Rightarrow Q$ satisfying axiom (A1).  If (A2) holds for all fibrant objects in $\M$, then (A2) holds for all objects in $\M$.
\end{cor}

\begin{proof}
Let $X$ be an object in $\M$ and let $\beta \colon X\xrightarrow{\simeq} X'$ be a fibrant replacement.  By (A1), we know that $Q\beta \colon QX\rightarrow QX'$ and $Q^2\beta \colon Q^2X\rightarrow Q^2X'$ are weak equivalences.  Consider the commutative diagram 
\vspace{-3mm}
\begin{center}
\begin{tikzpicture}[node distance=3cm, auto, scale=1]

\node (A) {$QX$};
\node (B) [right of=A] {$QX'$};
\node (C) [node distance=1.8cm, below of=A] {$Q^2X$};
\node (D) [node distance=1.8cm, below of=B] {$Q^2X'$.};

\draw[->] (A) to node {$Q\beta$} (B);
\draw[->] (A) to node [swap] {$Q\eta_X$} (C);
\draw[->] (C) to node [swap] {$Q^2 \beta$} (D);
\draw[->] (B) to node {$Q \eta_{X'}$} (D);

\end{tikzpicture}
\end{center}
\vspace{-2mm}
obtained by applying $Q$ to the naturality diagram for $\eta$. 
Since $X'$ is fibrant, the right vertical arrow is a weak equivalence by our assumption that (A2) holds for fibrant objects.  As the horizontal arrows are also weak equivalences, we see that $Q\eta_X$ is also a weak equivalence. The fact that $\eta_{QX}$ is a weak equivalence is proved in a similar fashion by replacing the vertical maps in the diagram above with $\eta_{QX}$ and $\eta_{QX'}$, respectively. 
\end{proof}

\begin{cor} \label{simplifiedA3} 
Let $\M$ be a right proper model category together with an endofunctor $Q \colon \M \rightarrow \M$ and natural transformation $\eta \colon \id \Rightarrow Q$ satisfying axioms (A1) and (A2) above. If
\begin{itemize}
    \item[(A3')] $Q$ preserves homotopy pullback squares,
\end{itemize} 
then axiom (A3) holds, and in particular the model structure $\M_Q$ exists and is right proper.
\end{cor}

\begin{proof} 
We note that if $Q$ preserves fibrations, this statement is a direct consequence of the right properness condition. To see why it holds more generally, assume that $Q$ preserves homotopy pullback squares and that we have a commutative square as in Theorem \ref{BousfieldQ} satisfying the hypotheses of (A3).  Since one of those hypotheses is that $h \colon X\rightarrow Y$ is a fibration, and $\M$ is right proper, such a square is necessarily a homotopy pullback square by Proposition \ref{rightproperpullback}.  By (A3'), the diagram
\vspace{-2mm}
\begin{center}
\begin{tikzpicture}[node distance=2.5cm, auto, scale=1]

\node (A) {$QV$};
\node (B) [right of=A] {$QX$};
\node (C) [node distance=1.7cm, below of=A] {$QW$};
\node (D) [node distance=1.7cm, below of=B] {$QY$};

\draw[->] (A) to node {$Qf$} (B);
\draw[->] (A) to node {} (C);
\draw[->] (C) to node [swap] {$Qg$} (D);
\draw[->] (B) to node {$Qh$} (D);

\end{tikzpicture}
\end{center}
\vspace{-2mm}
is a homotopy pullback square.  Again by the hypotheses of (A3), $Qg$ is a weak equivalence, and hence $Qf$ is a weak equivalence by Proposition \ref{properwe}, establishing (A3).
\end{proof} 

For ease of reference, we make the following definition. 

\begin{definition} \label{Bousenddef}
    A \emph{Bousfield endofunctor} $Q \colon \mathcal M \rightarrow \mathcal M$ is an endofunctor together with a natural transformation $\eta \colon \id \Rightarrow Q$ satisfying axioms (A1) and (A2) from Theorem \ref{BousfieldQ} and axiom (A3') from Corollary \ref{simplifiedA3}.
\end{definition}

We refer to the fibrations in $\M_Q$, which are completely determined by the cofibrations and weak equivalences as described in Theorem \ref{BousfieldQ}, as $Q$-\emph{fibrations}.  We include the following useful characterizations of them, the first of which was part of Bousfield's original statement of Theorem \ref{BousfieldQ}.  The second part of this result is a consequence of the first via Proposition \ref{properwe}. Some similar results are proved in \cite[X.4]{gj} under the further assumption of left properness. 

\begin{prop} \cite[9.3]{bousfield} \label{Qfibrations} \label{QpresWE}
Let $\M$ be a right proper model category together with an endofunctor $Q \colon \M \rightarrow \M$ and natural transformation $\eta \colon \id \Rightarrow Q$ satisfying axioms (A1), (A2), and (A3).
\begin{enumerate}[(1)]
\item A map $f \colon X \rightarrow Y$ in $\M$ is a $Q$-fibration if and only if it is a fibration in $\M$ and the induced diagram
\vspace{-2mm}
\begin{center}
\begin{tikzpicture}[node distance=2.2cm, auto, scale=1]

\node (A) {$X$};
\node (B) [right of=A] {$QX$};
\node (C) [node distance=1.7cm, below of=A] {$Y$};
\node (D) [node distance=1.7cm, below of=B] {$QY$};

\draw[->] (A) to node {$\eta_X$} (B);
\draw[->] (A) to node [swap] {$f$} (C);
\draw[->] (C) to node [swap] {$\eta_Y$} (D);
\draw[->] (B) to node {$Qf$} (D);

\end{tikzpicture}
\end{center}
\vspace{-2mm}
is a homotopy pullback square in $\M$. 

\item Assume that $\eta_\ast \colon \ast\rightarrow Q(\ast)$ is a weak equivalence where $\ast$ is the terminal object in $\M$. Then an object $X$ of $\M$ is $Q$-fibrant if and only if it is fibrant in $\M$ and the map $\eta_X \colon X \rightarrow QX$ is a weak equivalence in $\M$.
\end{enumerate}
\end{prop}

We refer to the homotopy pullback condition in the first part of the corollary as the $Q$-\emph{fibration condition}, for ease of referring to it in later sections.

\section{The homotopy functor model structure} \label{fcexamples} \label{s:nexmodel} 

In this section, we apply Theorem \ref{BousfieldQ} to obtain a model structure on $\Fun(\C,\D)$ in which the fibrant objects are homotopy functors, or functors that preserve weak equivalences.  This model structure was established in \cite{br}, but we treat it in some detail so that we can refer to the techniques in more generality. While it is a first step in developing a model structure for $n$-excisive functors, this example is also key because it requires us to look closely at many of the ingredients used.   

We begin by identifying the necessary assumptions on the categories $\C$ and $\D$ for such a model structure to exist, and one of those conditions requires the existence of a well-behaved simplicial fibrant replacement functor.  While fibrant replacements exist in any model category, and they can be taken to be functorial in nice cases such as cofibrantly generated model structures, it is important for our purposes that they also preserve the simplicial structure. So, our first goal is to describe such a functor, for which we recall the following definition.  Although we assume we are working in a model category, since that is the context in which we use it, this definition can be applied to any cocomplete category. 

\begin{definition} \label{soadefn} 
Let $\M$ be a model category. A set $J$ of maps in $\M$ \emph{permits the small object argument} if there exists a cardinal $\kappa$ such that for every regular cardinal $\lambda \geqslant \kappa$ and every $\lambda$-sequence 
\begin{center}
\begin{tikzpicture}[node distance=1.8cm, auto]

\node (A) {$A_0$};
\node (B) [right of=A] {$A_1$};
\node (C) [right of=B] {$A_2$};
\node (D) [right of=C] {$\hdots$};
\node (E) [right of=D] {$A_{\beta}$};
\node (F) [right of=E] {$\hdots$};
\node (G) [right of=F] {$(\beta < \lambda)$};

\draw[->] (A) to node {} (B);
\draw[->] (B) to node {} (C);
\draw[->] (C) to node {} (D);
\draw[->] (D) to node {} (E);
\draw[->] (E) to node {} (F);

\end{tikzpicture} 
\end{center}
in $\B$ such that each $A_{\beta} \rightarrow A_{\beta + 1}$ is a transfinite composition of pushouts of maps in $J$, the map 
\vspace{-2mm}
\begin{center}
\begin{tikzpicture}[node distance=5cm, auto]
\node (A) {$\underset{\beta < \lambda}{\colim} \, \Hom_{\M} (U, A_{\beta})$};
\node (B) [right of=A] {};

\node (c) [node distance=1mm, above of=B] {$\Hom_{\M} \left(U, \underset{\beta < \lambda}{\colim} \, A_{\beta}\right)$};
\node (d) [node distance=1.9cm, left of=B] {};
\draw[transform canvas={yshift=0.6ex}, ->] (A) to node {} (d);

\end{tikzpicture} 
\end{center}
is an isomorphism whenever $U$ is the domain of some map in $J$.
\end{definition}

\begin{remark} \label{remark:Small object argument} 
Note that if $J$ permits the small object argument and $U$ is in the domain of some map in $J$, then any map $U \rightarrow \underset{\beta < \lambda}{\colim} \, A_{\beta}$ factors through some $A_{\alpha}$. 
\end{remark}

The following result is due to Shulman \cite[24.2]{shulman}, but since this case is not proved in full detail in that paper, and we need to use some of the techniques elsewhere, we give a brief outline of the proof here.

\begin{theorem} \label{SimpFibRep}
There exists a simplicial fibrant replacement functor on any cofibrantly generated simplicial model category. 
\end{theorem}

\begin{proof}
Let $\M$ be a cofibrantly generated simplical model category, and let $J$ denote a set of generating acyclic cofibrations. For any object $A$ of $\M$, define $A_0 = A$ and, for each $\beta$, define $A_{\beta+1}$ via the pushout
\vspace{-1mm}
\begin{equation} \label{fibreppushout}
\begin{tikzpicture}[baseline=(current bounding box.center), node distance=4.5cm, auto, scale=1]

\node (A) {$\displaystyle \coprod_{U \longrightarrow V} \Map_{\M}(U,A_{\beta}) \otimes U$};
\node (B) [transform canvas={yshift=0.6ex}, right of=A] {$A_{\beta}$};
\node (C) [node distance=1.8cm, below of=A] {$\displaystyle \coprod_{U \longrightarrow V} \Map_{\M}(U,A_{\beta}) \otimes V$};
\node (D) [transform canvas={yshift=0.6ex}, right of=C] {$A_{\beta+1}$};

\draw[transform canvas={yshift=0.8ex}, ->] (A) to node {} (B);
\draw[transform canvas={yshift=0.8ex}, ->] (C) to node [swap] {$\iota_{\beta}^A$} (D);
\draw[transform canvas={yshift=0.7ex}, ->] (A) to node {} (C);
\draw[->] (B) to node {$\varphi_{\beta}^A$} (D);

\end{tikzpicture} 
\end{equation}
where the coproducts are taken over the set of generating acyclic cofibrations. We define the fibrant replacement of $A$ to be the transfinite sequential colimit $\mathbb{F}A := \colim_{\beta} A_{\beta}$. By setting $K = \varnothing$ and $L = \Map_{\M}(U,A_{\beta})$ in (S7') of Remark \ref{tensorcotensorrmk}, we see that each $\Map_{\M}(U,A_{\beta}) \otimes U \rightarrow \Map_{\M}(U,A_{\beta}) \otimes V$ is an acyclic cofibration. Then, using the fact that acyclic cofibrations are closed under small coproducts and are stable under pushouts, we see that each $\varphi_{\beta}^A \colon A_{\beta} \rightarrow A_{\beta + 1}$ is an acyclic cofibration. It follows that the composite $\varphi_A \colon A = A_0 \rightarrow \colim_{\beta} A_{\beta} = \mathbb{F}A$ is also an acyclic cofibration; see \cite[10.3.4]{hirschhorn}.

It is a now a straightforward exercise to show that the unique map $\mathbb{F}A \rightarrow \ast_{\M}$ is a fibration, that $\mathbb{F}$ is a simplicial functor, and that this factorization is functorial.
\end{proof}

\begin{remark}\label{FibRepRemarks} 
Observe that for any cofibrantly generated simplicial model category $\M$, the functorial fibrant replacement functor satisfies the following properties.
\begin{enumroman}
\item If $A$ is cofibrant, then $\mathbb{F}A$ is cofibrant, following from the fact that $A \rightarrow \mathbb{F}A$ is an acyclic cofibration.

\item If $A \rightarrow B$ is a weak equivalence, then $\mathbb{F}A \rightarrow \mathbb{F}B$ is a weak equivalence by the two-out-of-three property.

\item It follows from (i) and (ii) that if $A \rightarrow B$ is a weak equivalence between cofibrant objects, then $\mathbb{F}A \rightarrow \mathbb{F}B$ is a weak equivalence between objects that are both fibrant and cofibrant, and is therefore a simplicial homotopy equivalence.

\item Since fibrant replacements are sequential colimits they commute with all colimits in $\M$, i.e., $\mathbb{F} (\underset{i}{\colim} \, A_i ) \cong \underset{i}{\colim} \, \mathbb{F}A_i$. 
\end{enumroman}
\end{remark}

The remaining piece that we need before stating our assumptions is the following definition.

\begin{definition} \label{SimpFinPres}
An object $C$ in a simplicial category $\C$ is \emph{simplicially finitely presentable} if the representable functor $R^C = \Map(C,-) \colon \C \rightarrow \s$ preserves filtered colimits, so
\[ \Map_{\C}(C, \underset{i}{\colim} \, A_i) \cong \underset{i}{\colim} \, \Map_{\C}(C,A_i).\]
\end{definition}

\begin{convention} \label{HomotopyConventions} 
In this subsection, we make the following assumptions in addition to those listed in Convention \ref{convention}.
\begin{enumerate}
\item \label{cofibpresentable} 
Suppose that $\C$ is a full simplicial subcategory of a cofibrantly generated simplicial model category $\B$, and that the objects of $\C$ are all cofibrant and simplicially finitely presentable. 

\item \label{seqcolimits} 
For every object $C$ of $\C$, the fibrant replacement $\mathbb{F}C$ is a sequential colimit of objects $C_{\beta}$ in $\C$. 

\item \label{preseqcolims}
Weak equivalences, fibrations, and homotopy pullbacks are preserved under sequential colimits in $\D$. 
\end{enumerate}
\end{convention}

\begin{remark} 
Although this list may seem lengthy, we claim that these conditions are satisfied by many familiar model categories. For instance, condition \eqref{cofibpresentable} can be satisfied by taking $\C$ to be the full subcategory of $\B$ consisting of the cofibrant simplicially finitely presentable objects.

If we only asked for the objects $C_\beta$ to be in $\B$, then condition \eqref{seqcolimits} is satisfied for any cofibrantly generated model category $\B$, given our construction of fibrant replacements via the small object argument. The issue is whether these objects are in the subcategory $\C$, not in the larger model category $\B$.  It does hold in many examples of interest, for example taking $\B$ to be the usual model structure on topological spaces or simplicial sets and $\C$ the subcategory of finite spaces or finite simplicial sets, respectively. 

Similarly, condition \eqref{preseqcolims} holds in many nice cases, such as the categories of topological spaces and simplicial sets.  For conditions under which sequential colimits preserve weak equivalences and fibrations, see \cite[\S 7.4]{hovey}; for a discussion of filtered colimits commuting with homotopy pullbacks, see \cite{hirschnotes}.
\end{remark}

The assumptions on $\D$ have the following consequence.

\begin{lemma} \label{seqcolimandhocolim} 
Sequential colimits are weakly equivalent to sequential homotopy colimits in $\D$. 
\end{lemma}

\begin{proof}
Let $\colim D_i$ be a sequential colimit in $\D$ and let $D'_0$ be a cofibrant replacement of $D_0$. We can factor the resulting map $D'_0 \rightarrow D_1$ as a cofibration $D'_0 \rightarrow D'_1$ followed by an acyclic fibration $D'_1 \rightarrow D_1$. We can then repeat the process with the map $D'_1 \rightarrow D_2$, and so on, giving a commutative diagram 
\vspace{-2mm}
\begin{center}
\begin{tikzpicture}[node distance=2cm, auto]

\node (A) {$D'_0$};
\node (B) [right of=A] {$D'_1$};
\node (C) [right of=B] {$D'_2$};
\node (D) [right of=C] {$\hdots$};
\node (E) [right of=D] {$\colim D'_i$};

\node (F) [node distance=1.5cm, below of=A] {$D_0$};
\node (G) [right of=F] {$D_1$};
\node (H) [right of=G] {$D_2$};
\node (I) [right of=H] {$\hdots$};
\node (J) [right of=I] {$\colim D_i$};

\draw[right hook->] (A) to node {} (B);
\draw[right hook->] (B) to node {} (C);
\draw[right hook->] (C) to node {} (D);
\draw[right hook->] (D) to node {} (E);

\draw[->] (F) to node [swap] {} (G);
\draw[->] (G) to node [swap] {} (H);
\draw[->] (H) to node [swap] {} (I);
\draw[->] (I) to node {} (J);

\draw[->>] (A) to node [swap] {$\simeq$} (F);
\draw[->>] (B) to node {$\simeq$} (G);
\draw[->>] (C) to node {$\simeq$} (H);
\draw[dashed,->] (E) to node {} (J);

\end{tikzpicture} 
\end{center}
where the map $\colim D'_i \rightarrow \colim D_i$ is induced by the universal property of colimits and, by Convention \ref{HomotopyConventions}(\ref{preseqcolims}), is a weak equivalence. Since the upper horizontal arrows in the diagram are all cofibrations between cofibrant objects, we have $\hocolim D_i = \colim D'_i$, proving the result. 
\end{proof}

Before proceeding to the homotopy functor model structure, we need to state one more consequence of our conventions.  Recall that the simplicial left Kan extension of $F \colon \C \rightarrow \D$ along a simplicial functor $p \colon \C \rightarrow \C'$ is another simplicial functor $p_!(F) \colon \C' \rightarrow \D$ together with a simplicial natural transformation 
\vspace{-2mm}
\begin{center}
\begin{tikzpicture}[node distance=2.8cm, auto]

\node (A) {$\C$};
\node (B) [right of=A] {$\D$};
\node (C) [node distance=1.4cm, right of=A, below of=A] {$\C'$};

\draw[->] (A) to node {$F$} (B);
\draw[->] (A) to node [swap] {$p$} (C);
\draw[->] (C) to node [swap] {$p_!(F)$} (B);

\node () [node distance=9mm, above of=C] {$\Downarrow \eta$};

\end{tikzpicture} 
\end{center}
that is appropriately universal with respect to this property. The functor $p_!(F)$ is given by the coend
\[ p_!(F)(B) = \int^{A} \Map_{\C'}(p(A), B) \otimes FA \]
in $\D$ \cite[2.4]{kelly}.  

\begin{remark}\label{idkan}
Observe that for the inclusion functor $i \colon \C \rightarrow \B$, it follows from Lemma \ref{Yoneda} that 
\[ i_!(F)(C) = \int^{A} \Map_{\B}\left(i(A), C\right) \otimes FA = \int^{A} \Map_{\C}(A, C) \otimes FA \cong FC \]
for any object $C$ of $\C$, so $i_!(F) \circ i = F$, and $\eta$ is the identity transformation.
\end{remark}

\begin{prop} \label{KanPresFiltered}
For any simplicial functor $F \colon \C \rightarrow \D$, the simplicial left Kan extension $i_!(F)$ of $F$ along the inclusion $i \colon \C \rightarrow \B$ preserves filtered colimits.
\end{prop}

\begin{proof}
Let $\underset{k}{\colim} \, B_k$ be a filtered colimit in $\B$. Then
\[ \begin{aligned}
i_!(F)\left(\underset{k}{\colim} \,  B_k\right) &= \int^{A} \Map_{\B}\left(A, \underset{k}{\colim} \, B_k \right) \otimes FA \\
&\cong \int^{A} \underset{k}{\colim} \left( \Map_{\B}(A, B_k) \right) \otimes FA\\
&\cong \underset{k}{\colim} \  \int^{A} \Map_{\B}(A, B_k) \otimes FA\\
&= \underset{k}{\colim} \  i_!(F)(B_k),
\end{aligned} \]
where the first isomorphism holds by Convention \ref{HomotopyConventions}\eqref{cofibpresentable}, and the second isomorphism holds since both tensoring and coends commute with colimits; the former is a left adjoint and the latter is itself a colimit.
\end{proof}

We now have the ingredients to define a means for replacing any functor $F \colon \C \rightarrow \D$ by a homotopy functor. This definition was first given in \cite[4.10]{br}.

\begin{definition}\label{hfdefn} 
Let $F \colon \C \rightarrow \D$ be a simplicial functor. We define the simplicial functor $F^{\hf} \colon \C \rightarrow \D$ as the composite  
\begin{center}
\begin{tikzpicture}[node distance=2.2cm, auto, scale=1]

\node (A) {$\C$};
\node (B) [right of=A] {$\B$};
\node (C) [right of=B] {$\B$};
\node (D) [right of=C] {$\D$,};

\draw[right hook->] (A) to node {$i$} (B);
\draw[->] (B) to node {$\mathbb{F}$} (C);
\draw[->] (C) to node {$i_!(F)$} (D);

\end{tikzpicture}
\end{center}
where $i \colon \C \rightarrow \B$ is the inclusion and $i_!(F)$ is the simplicial left Kan extension.
\end{definition}

\begin{remark}\label{varphi induces theta}
Observe that since $i_!(F) \circ i = F$, there exists a canonical natural transformation $\theta_F \colon F \Rightarrow F^{\hf}$ induced by the transformation $\varphi \colon \id_{\B} \Rightarrow \mathbb{F}$ defined in the proof of Theorem \ref{SimpFibRep}. 
\end{remark}

Let us first verify that this construction does in fact produce a homotopy functor, as suggested by the notation.

\begin{prop} \label{hfishf}
If $F \colon \C \rightarrow \D$ is a simplicial functor, then $F^{\hf} \colon \C \rightarrow \D$ is a simplicial homotopy functor.
\end{prop}

\begin{proof}
Consider a weak equivalence $f \colon A \rightarrow B$ in $\C$.  Then, by Remark \ref{FibRepRemarks}(iii) and Convention \ref{HomotopyConventions}(\ref{cofibpresentable}), $\mathbb{F}(f) \colon \mathbb{F}(A) \rightarrow \mathbb{F}(B)$ is a simplicial homotopy equivalence. Since simplicial homotopy equivalences are preserved by any simplicial functor, $i_!(F) \left(\mathbb{F}(f) \right) = F^{\hf}(f)$ is a simplicial homotopy equivalence, and therefore a weak equivalence.  
\end{proof}

In light of Proposition \ref{QpresWE}, the following fact is useful.

\begin{lemma} \label{hf pres terminal}
The endofunctor $(-)^{\hf} \colon \Fun(\C, \D) \rightarrow \Fun(\C,\D)$ preserves the terminal object.
\end{lemma}

\begin{proof}
Consider an object $C$ of $\C$. By Convention \ref{HomotopyConventions}(\ref{seqcolimits}) we can express $\mathbb{F}C$ as a sequential colimit of objects $C_{\beta}$ in $\C$. Then 
\[ F^{\hf}(C) = \left(i_!(F) \circ \mathbb{F}\right)(C) = i_!(F) (\underset{{\beta}}{\colim} \ C_{\beta}) \cong \underset{{\beta}}{\colim}\ i_!(F) \ (C_{\beta}), \]
where the isomorphism holds by Proposition \ref{KanPresFiltered}. Now since each $C_{\beta}$ is an object of $\C$, we have 
\[ \underset{{\beta}}{\colim} \ i_!(F) (C_{\beta}) = \underset{\beta}{\colim} \ FC_{\beta}\]
by Remark \ref{idkan}. Therefore, if $FC_{\beta}$ is the terminal object for all ${\beta}$, then $F^{\hf}(C)$ is also terminal.
\end{proof}

We now recall the $\hf$-model structure \cite[4.14]{br}, which can be proved by showing that $(-)^{\hf}$ satisfies axioms (A1), (A2), and (A3'). 

\begin{theorem} \label{hfmodel}
Under the assumptions of Convention \ref{HomotopyConventions} the category $\Fun(\C,\D)$ has a simplicial right proper model structure, denoted by $\Fun(\C,\D)_{\hf}$, in which $F \rightarrow G$ is a weak equivalence if $F^{\hf} \rightarrow G^{\hf}$ is a levelwise weak equivalence, and the cofibrations are precisely the projective cofibrations.  
\end{theorem}

We conclude this section with a few consequences of this result.  They are all expected properties for localized model structures, and whose analogues are known for left Bousfield localizations, but that do not appear to be known in generality for Bousfield-Friedlander localization.  

\begin{prop} \label{fibranthf} 
Let $\alpha \colon F \rightarrow G$ be a fibration in $\Fun(\C,\D)$. Then $\alpha$ is a fibration in $\Fun(\C,\D)_{\hf}$ if and only if for each weak equivalence $A \rightarrow B$ in $\C$ the diagram
\vspace{-1mm}
\begin{center}
\begin{tikzpicture}[node distance=1.9cm, auto, scale=1]

\node (A) {$FA$};
\node (B) [right of=A] {$FB$};
\node (C) [node distance=1.5cm, below of=A] {$GA$};
\node (D) [node distance=1.5cm, below of=B] {$GB$};

\draw[->] (A) to node {} (B);
\draw[->] (A) to node [swap] {} (C);
\draw[->] (C) to node [swap] {} (D);
\draw[->] (B) to node {} (D);

\end{tikzpicture}
\end{center}
is a homotopy pullback square.  Moreover, a functor $F \colon \C \rightarrow \D$ is fibrant in $\Fun(\C,\D)_{\hf}$ if and only if it is a levelwise fibrant homotopy functor. 
\end{prop} 

\begin{proof}
The first statement can be proved just as in \cite[4.15]{br} using Convention \ref{HomotopyConventions}, Proposition \ref{Qfibrations}, and the fact that the maps $A_i\rightarrow A_{i+1}$ in the sequential colimit defining the fibrant replacement of Theorem \ref{SimpFibRep} are acyclic cofibrations. The statement about $\hf$-fibrant objects follows from the first statement and Proposition \ref{properwe}.
\end{proof}

The next result is useful for establishing that our model for an $n$-excisive approximation to a functor agrees with that of Goodwillie for all homotopy functors, which we prove in Section \ref{n-excisive model}.  In particular, we prove that this result holds for homotopy functors that need not be levelwise fibrant.

\begin{prop}\label{hfpreshf}
If $F$ is a homotopy functor, then $\theta_F \colon F \Rightarrow F^{\hf}$ is a levelwise weak equivalence.
\end{prop}

\begin{proof}
For any object $C$ of $\C$, we have $F^{\hf}(C) = \colim_\beta FC_{\beta}$ by Convention \ref{HomotopyConventions}\eqref{seqcolimits}, Remark \ref{idkan}, and Proposition \ref{KanPresFiltered}. We can then write the map $\theta_{FC}$ as
\[ FC = \colim_\beta FC \longrightarrow \colim_\beta FC_{\beta}, \]
so in particular, $\theta_{FC}$ is the sequential colimit of the morphisms $F(\varphi_{\beta-1}^C \cdot \ldots \cdot \varphi_1^C \cdot \varphi_0^C) \colon FC \rightarrow FC_{\beta}$.  From the proof of Theorem \ref{SimpFibRep}, we know that each $\varphi^C_i$ is a weak equivalence.  Since $F$ is a homotopy functor, it follows from Convention \ref{HomotopyConventions}\eqref{preseqcolims} that $\theta_{FC}$ is a weak equivalence. 
\end{proof}

\section{The $n$-excisive model structure} \label{n-excisive model}

In this section, we establish the existence of a model structure on the category $\Fun(\C,\D)$ whose fibrant objects are the $n$-excisive functors for a given $n$. To ensure that the endofunctor $\widehat{P}_n$ we use to obtain this model structure satisfies axiom (A1) of Theorem \ref{BousfieldQ}, we make use of the functor $(-)^{\hf}$ of Definition \ref{hfdefn}.   So, in addition to the assumptions that we made about the category $\D$ for the $\hf$-model structure, we need to add another mild hypothesis to ensure that  $\widehat{P}_n$ interacts nicely with the $\hf$-model structure.  Thus, we begin this section by making the following definition and then establishing some results for the $\hf$-model structure that we need.

\begin{definition}
An object $U$ of a category $\D$ is \emph{finite relative to a subcategory} $\A$ if, for all limit ordinals $\lambda$ and $\lambda$-sequences  
\begin{center}
\begin{tikzpicture}[node distance=1.8cm, auto]

\node (A) {$A_0$};
\node (B) [right of=A] {$A_1$};
\node (C) [right of=B] {$A_2$};
\node (D) [right of=C] {$\hdots$};
\node (E) [right of=D] {$A_{\beta}$};
\node (F) [right of=E] {$\hdots$};
\node (G) [right of=F] {$(\beta < \lambda)$};

\draw[->] (A) to node {} (B);
\draw[->] (B) to node {} (C);
\draw[->] (C) to node {} (D);
\draw[->] (D) to node {} (E);
\draw[->] (E) to node {} (F);

\end{tikzpicture} 
\end{center}
in $\D$ such that each $A_{\beta} \rightarrow A_{\beta + 1}$ is in $\A$, the map 
\begin{center}
\begin{tikzpicture}[node distance=4.5cm, auto]

\node (A) {$\colim_\beta \Hom_{\D} (U, A_{\beta})$};
\node (B) [right of=A] {$\Hom_{\D} \left(U, \colim_\beta A_{\beta}\right)$};

\draw[transform canvas={yshift=0.6ex}, ->] (A) to node {} (B);

\end{tikzpicture} 
\end{center}
is an isomorphism. 
\end{definition}
For this section, we assume the following.
\begin{convention} \label{nexcisiveconv}
The set $J$ of generating cofibrations of $\D$ can be chosen such that the codomain of each map is finite relative to $J$, in addition to Conventions \ref{GoodwilleConventions} and \ref{HomotopyConventions}. 
\end{convention}

\begin{remark} 
Note that since $\D$ is cofibrantly generated by assumption, the set $J$ of generating cofibrations can always be chosen such that the domain of each map is finite relative to $J$. The extra condition on the codomain is satisfied, for example, by any finitely generated model category, as described in \cite{hovey}.
\end{remark}

Our assumptions on $\D$ guarantee the following result. The proof is the same as the one for finitely generated model categories, see \cite[7.4.1]{hovey}.

\begin{lemma} \label{FiniteGenLemma}
Suppose that $\D$ satisfies the conditions of Convention \ref{nexcisiveconv}, $\lambda$ is an ordinal, $X,Y \colon \lambda \rightarrow \D$ are $\lambda$-sequences of acyclic cofibrations, and $p \colon X \rightarrow Y$ is a natural transformation such that $p_{\beta} \colon X_{\beta} \rightarrow Y_{\beta}$ is a fibration for all $\beta < \lambda$. Then $\colim \, p_{\beta} \colon \colim \, X_{\beta} \rightarrow \colim \, Y_{\beta}$ is a fibration that is a weak equivalence if each $p_\beta$ is.
\end{lemma}

Now we establish the compatibility with the functor $(-)^{\hf}$ that we need.

\begin{prop} \label{hfpresproj} 
The endofunctor $(-)^{\hf}$ on $\Fun(\C,\D)$ preserves levelwise fibrations. 
\end{prop}

\begin{proof}
Let $\alpha \colon F \rightarrow G$ be a levelwise fibration. Since $F^{\hf}(C) = \colim_\beta FC_{\beta}$ for any object $C$ in $\C$, the component $\alpha^{\hf}_C \colon F^{\hf}(C) \rightarrow G^{\hf}(C)$ at $C$ is induced by the morphisms $\alpha_{C_{\beta}} \colon FC_{\beta} \rightarrow GC_{\beta}$ for all $\beta$.  Since each of these maps $\alpha_{C_\beta}$ is a fibration in $\D$ by assumption, and the maps in the colimit sequence defining $\colim_\beta \, FC_{\beta}$ are all acyclic cofibrations as explained in the proof of Theorem \ref{SimpFibRep}. 
\end{proof}

We now turn our attention to $n$-excisive approximations of functors, using Definition \ref{Pndefn}.  In addition, we replace $F$ with $F^{\hf}$ as defined in Definition \ref{hfdefn}, making it possible to apply some results from \cite{goodwillie3} that only hold for homotopy functors.  

\begin{definition}\label{Pnhdefn}
For a functor $F$ in $\Fun(\C,\D)$, we define the functor $\widehat P_nF$ by
\[ \widehat P_nF:=\colim_k \left(F^{\hf}\rightarrow T_n (F^{\hf})\rightarrow T_n^2(F^{\hf}) \rightarrow \dots \rightarrow T_n^k(F^{\hf})\rightarrow \dots \right). \] 
Using $T_n^\infty$ to represent this colimit, we have 
\[ \widehat P_nF=T_n^{\infty}F^{\hf}. \]
There is a morphism $F\rightarrow \widehat P_nF$ given by the composite
\begin{center}
\begin{tikzpicture}[node distance=2.2cm, auto, scale=1]

\node (A) {$F$};
\node (B) [right of=A] {$F^{\hf}$};
\node (C) [node distance=3.5cm, right of=B] {$T_n^\infty(F^{\hf}) = \widehat P_nF$,};

\draw[->] (A) to node {$\theta_F$} (B);
\draw[->] (B) to node {$\iota_{F^{\hf}}$} (C);

\end{tikzpicture}
\end{center}
where $\iota$ is the natural transformation whose component at $G$ is induced by the natural maps $G\rightarrow T_nG\rightarrow T_n^2G\dots$.  Since this construction is natural in $F$, it induces a natural transformation  $\widehat p_n \colon \id_{\Fun(\C,\D)}\Rightarrow \widehat P_n$.  We sometimes omit the subscript on $\widehat p_n$ when it can be understood from context.
\end{definition}

We note that by Lemma \ref{seqcolimandhocolim}, $\widehat P_nF$ is weakly equivalent to $P_n(F^{\hf})$.  One can prove that $\widehat P_nF$ is an $n$-excisive functor using a proof similar to that of \cite[1.8]{goodwillie3}. The argument uses the fact that $F^{\hf}$ is a homotopy functor and is the reason we replace $F$ with $F^{\hf}$ in the definition of $\widehat P_nF$. By Lemma \ref{seqcolimandhocolim} and Proposition \ref{hfpreshf}, $\widehat P_nF$ is weakly equivalent to Goodwillie's construction of $P_nF$ when $F$ is a homotopy functor.

\begin{prop}\label{Pnpreshf} 
The endofunctor $\widehat P_n$ on the category $\Fun(\C,\D)$ preserves $\hf$-fibrations.
\end{prop} 

\begin{proof} 
By Proposition \ref{Qfibrations}, it suffices to prove that if $F\rightarrow G$ is an $\hf$-fibration, then $\widehat P_nF\rightarrow \widehat P_nG$ is a projective fibration and 
\begin{center} 
\begin{tikzpicture}[node distance=2.2cm, auto, scale=1]\label{Pnfibcube}

\node (A) {$\widehat P_nF$};
\node (B) [right of=A] {$(\widehat P_nF)^{\hf}$};
\node (C) [node distance=1.5cm, below of=A] {$\widehat P_nG$};
\node (D) [node distance=1.5cm, below of=B] {$(\widehat P_nG)^{\hf}$};

\draw[->] (A) to node {} (B);
\draw[->] (A) to node [swap] {} (C);
\draw[->] (C) to node [swap] {} (D);
\draw[->] (B) to node {} (D);

\end{tikzpicture}
\end{center}
is a homotopy pullback square. The fact that $\widehat P_nF\rightarrow \widehat P_nG$ is a projective fibration follows from the definition of $\widehat P_n$, Proposition \ref{hfpresproj}, the fact that homotopy limits preserve fibrations, and Convention \ref{HomotopyConventions}(\ref{preseqcolims}). To confirm that the diagram above is a homotopy pullback square, recall that $\widehat P_nF$ and $\widehat P_nG$ are homotopy functors by Proposition \ref{hfishf}. It follows that the horizontal maps are weak equivalences by Proposition \ref{hfpreshf}, and the square is a homotopy pullback by Proposition \ref{properwe}.
\end{proof}

With the preceding proposition and definitions in place, we can establish the existence of the $n$-excisive model structure.  This theorem is also proved by Biedermann and R\"ondigs in \cite[5.8]{br} using a different, but naturally equivalent, model for the $n$-excisive approximation of a functor.  To obtain cofibrant generation of this model structure in Section \ref{s:nexccof}, we need to take a localization of the $\hf$-model structure, but arguments similar to the ones presented here can be used to place an $n$-excisive model structure on $\Fun(\C,\D)$ with the projective model structure instead.

\begin{theorem}\label{t:nexcisive}
Under the assumptions of Convention \ref{nexcisiveconv}, the category $\Fun(\C, \D)$ has a right proper model structure in which a morphism $F \rightarrow G$ is a weak equivalence if $\widehat P_n F \rightarrow \widehat P_n G$ is a weak equivalence in the $\hf$-model structure, and is a cofibration precisely if it is a cofibration in the $\hf$-model structure.
\end{theorem}

\begin{proof}
We apply Corollary \ref{simplifiedA3} to $\Fun(\C,\D)_{\hf}$ and the endofunctor $\widehat{P}_n$.  In particular, we  verify axioms (A1), (A2), and (A3').

To prove  (A1), assume that $F\rightarrow G$ is an $\hf$-equivalence, i.e, that  $F^{\hf} \to  G^{\hf}$ is a levelwise weak equivalence.   Our choice of model for homotopy limits preserves such weak equivalences, so $T_n$ does as well.  Then Convention \ref{HomotopyConventions}\eqref{preseqcolims} guarantees that $\widehat P_n F \to \widehat P_n G$ is a levelwise weak equivalence.  Since the functor $(-)^{\hf}$ preserves weak equivalences (by the proof of Theorem \ref{hfmodel}), it follows that $\widehat{P}_nF\rightarrow \widehat{P}_nG$ is an $\hf$-equivalence.  

For (A2), we need to prove that the natural transformations $\widehat p_{ \widehat P_nF} \colon \widehat P_nF\rightarrow \widehat P_n\widehat P_nF$ and $\widehat P_n\widehat p_{F} \colon \widehat P_nF\rightarrow \widehat P_n\widehat P_n F$ are weak equivalences in the $\hf$-model structure.  Since $(-)^{\hf}$ preserves levelwise weak equivalences, it suffices to prove the stronger result that these natural transformations are levelwise weak equivalences.  

For any object $A$ in $\C$, the $(n+1)$-cubical diagram given by $U \mapsto A \ast U$ is a strongly homotopy cocartesian diagram, so applying any $n$-excisive functor $H$ to it produces a homotopy cartesian diagram. Then the map from the initial object in this diagram to the homotopy limit of the rest of the diagram is a levelwise weak equivalence.  By definition, this map is $H \rightarrow T_nH$, and since $T_n$ preserves levelwise weak equivalences, Convention \ref{HomotopyConventions}(\ref{preseqcolims}) guarantees that the natural transformation $\iota_H \colon H\rightarrow T_n^{\infty}H$ is a levelwise weak equivalence as well.

To see that $\widehat  p_{\widehat P_nF} \colon \widehat P_nF\rightarrow \widehat P_n\widehat P_nF$ is a levelwise weak equivalence, consider, for an arbitrary functor $X$, the commutative square
\vspace{-2mm}
\begin{center}
\begin{tikzpicture}[node distance=1.7cm, auto, scale=1]

\node (A) {$X$};
\node (B) [node distance=4cm, right of=A] {$X^{\hf}$};
\node (C) [below of=A] {$T_n^{\infty}X$};
\node (D) [below of=B] {$T_n^{\infty}X^{\hf}=\widehat P_nX.$};

\draw[->] (A) to node {$\theta_X$} (B);
\draw[->] (A) to node [swap] {$\iota_X$} (C);
\draw[->] (B) to node {$\iota_{X^{\hf}}$} (D);
\draw[->] (C) to node [swap] {$T_n^{\infty}\theta_X$} (D);
\end{tikzpicture}
\end{center}
\vspace{-1mm}
When $X=\widehat P_nF$, the horizontal maps are weak equivalences by Proposition \ref{hfishf} since $\widehat P_nF$ is a homotopy functor and $T_n^{\infty}$ preserves weak equivalences.  As noted above, the left vertical map is also a weak equivalence since  $\widehat P_nF$ is $n$-excisive. It follows that the composition of the right vertical and top horizontal arrows is a weak equivalence, but this composite is $\widehat p_{\widehat P_nF}$.

To prove that $\widehat P_n\widehat p_F$ is a levelwise weak equivalence, recall that it is the composite
\vspace{-2mm}
\begin{center}
\begin{tikzpicture}[node distance=4cm, auto, scale=1]

\node (A) {$\widehat P_nF$};
\node (B) [node distance=3cm, right of=A] {$\widehat P_n(F^{\hf})$};
\node (C) [right of=B] {$\widehat P_n(T_n^{\infty}(F^{\hf}))$};

\draw[->] (A) to node {$\widehat P_n\theta_F$} (B);
\draw[->] (B) to node {$\widehat P_n\iota_{F^{\hf}}$} (C);
\end{tikzpicture}
\end{center}
\vspace{-1mm}
The map $\widehat P_n\theta_F$ is a levelwise weak equivalence by Theorem \ref{hfmodel}, Convention \ref{HomotopyConventions}(\ref{preseqcolims}), and the fact that $T_n$ preserves weak equivalences.  To see that $\widehat P_n\iota_{F^{\hf}}$ is a levelwise weak equivalence, consider the commutative diagram
\vspace{-2mm}
\begin{center}
\begin{tikzpicture}[node distance=1.8cm, auto, scale=1]

\node (A) {$T_n^{\infty}(F^{\hf})$};
\node (B) [node distance=5cm, right of=A] {$T_n^{\infty}(T_n^{\infty}F^{\hf})$};
\node (C) [below of=A] {$T_n^{\infty}(F^{\hf})^{\hf}$};
\node (D) [below of=B] {$T_n^{\infty}(T_n^{\infty}F^{\hf})^{\hf}$};

\draw[->] (A) to node {$T_n^{\infty}(\iota_{F^{\hf}})$} (B);
\draw[->] (A) to node [swap] {$T_n^{\infty}(\theta_{F^{\hf}})$} (C);
\draw[->] (B) to node {$T_n^{\infty}(\theta_{T_n^{\infty}F^{\hf}})$} (D);
\draw[->] (C) to node [swap] {$T_n^{\infty}(\iota_{F^{\hf}})^{\hf}$} (D);
\end{tikzpicture}
\end{center}
\vspace{-1mm}
and note that the bottom horizontal arrow is $\widehat P_n\iota_{F^{\hf}}.$ The vertical maps are levelwise weak equivalences by Propositions \ref{hfishf} and \ref{hfpreshf}, and the fact that $T_n^{\infty}$ preserves weak equivalences.  We can use an argument similar to the one used in the last paragraph of the proof of \cite[1.8]{goodwillie3} to prove that the top horizontal map is a weak equivalence. In particular, it suffices to show that $T_n^{\infty}(F^{\hf})\rightarrow T_n^{\infty}(T_nF^{\hf})$ is a weak equivalence.  By Convention \ref{HomotopyConventions}\eqref{preseqcolims} and the fact that homotopy limits commute, the target of this map is weakly equivalent to $T_n T_n^{\infty} (F^{\hf})$.  However, $T_n^{\infty}(F^{\hf})$ is weakly equivalent to $T_nT_n^{\infty}(F^{\hf})$ because $T_n^\infty(F^{\hf})=\widehat P_nF$ is an $n$-excisive functor.  Hence, the bottom map in the diagram, which is $\widehat P_n\widehat p_{F}$, is a levelwise weak equivalence as well, completing the proof that (A2) holds.

Consider a homotopy pullback square
\vspace{-2mm}
\begin{center}
\begin{tikzpicture}[node distance=2cm, auto, scale=1]

\node (A) {$F$};
\node (B) [right of=A] {$G$};
\node (C) [node distance=1.6cm, below of=A] {$H$};
\node (D) [node distance=1.6cm, below of=B] {$K$,};

\draw[->] (A) to node {} (B);
\draw[->] (A) to node [swap] {} (C);
\draw[->] (C) to node [swap] {} (D);
\draw[->] (B) to node {} (D);

\end{tikzpicture}
\end{center}
in the $\hf$-model structure.  Since $(-)^{\hf}$ preserves levelwise weak equivalences,  it suffices to show that there is a levelwise weak equivalence from $\widehat P_nF$ to the homotopy pullback (in the $\hf$-model structure) of $\widehat P_nH\rightarrow \widehat P_nK\leftarrow \widehat P_nG$  to prove (A3').

We factor $H \to K$ into an $\hf$-equivalence $H\rightarrow H'$ followed by an $\hf$-fibration $H'\rightarrow K$ so that the homotopy pullback of $G\rightarrow H\leftarrow K$ in the $\hf$-model structure, which is right proper, is the strict pullback $H'\times _K G$. Note that $H'\times _K G$ is also the homotopy pullback of $H'\rightarrow K\leftarrow G$ in both the projective and $\hf$-model structures since every $\hf$-fibration is a projective fibration. 

Since $\widehat P_n$ preserves $\hf$-fibrations and $\hf$-weak equivalences (Proposition \ref{Pnpreshf} and axiom (A1)), $\widehat P_nH\rightarrow \widehat P_nH'\rightarrow \widehat P_nK$ is also a factorization via an $\hf$-equivalence and $\hf$-fibration.  As above, the strict pullback $\widehat P_nH'\times _{\widehat P_nK}\widehat P_nG$ is also the homotopy pullback of $\widehat P_nH\rightarrow \widehat P_nK\leftarrow \widehat P_nG$ in the $\hf$-model structure, and of $\widehat P_nH'\rightarrow \widehat P_nK\leftarrow \widehat P_nG$ in both model structures.  Hence, to complete the proof, it suffices to show there is a levelwise weak equivalence from $\widehat P_nF$ to the homotopy pullback of $\widehat P_nH'\rightarrow \widehat P_nK\leftarrow \widehat P_nG.$

We now consider the diagram
\vspace{-2mm}
\begin{equation}\label{H'square}
\begin{tikzpicture}[node distance=2cm, auto, scale=1]

\node (A) {$F$};
\node (B) [right of=A] {$G$};
\node (C) [node distance=1.6cm, below of=A] {$H'$};
\node (D) [node distance=1.6cm, below of=B] {$K$.};

\draw[->] (A) to node {} (B);
\draw[->] (A) to node [swap] {} (C);
\draw[->] (C) to node [swap] {} (D);
\draw[->] (B) to node {} (D);

\end{tikzpicture}
\end{equation}
By assumption, $F^{\hf}$ is levelwise equivalent to $(H'\times _K G)^{\hf}$. But $H'\times _K G$ is the  homotopy pullback of $H'\rightarrow K\leftarrow G$ in the projective model structure  and $(-)^{\hf}$ preserves such homotopy pullbacks, making $F^{\hf}$ levelwise equivalent to the homotopy pullback of $(H')^{\hf}\rightarrow K^{\hf}\leftarrow G^{\hf}$.  That is, applying $(-)^{\hf}$ to \eqref{H'square}
yields a homotopy pullback in the projective model structure.  Then Convention \ref{HomotopyConventions} and  the fact that $T_n$, as a homotopy limit, preserves homotopy pullback diagrams ensure that  applying $\widehat P_n$ to \eqref{H'square}
produces a homotopy pullback in the projective model structure.  
\end{proof}

Proposition \ref{QpresWE} allows us to characterize fibrations and fibrant objects in the $n$-excisive model structure as follows.  Biedermann and R\"ondigs provide a similar characterization of fibrations  in \cite[5.9]{br}.

\begin{prop} \label{fibrantexcisive}
A morphism $\alpha \colon F\rightarrow G$ in $\Fun(\C,\D)$ is a fibration in the $n$-excisive model structure if and only if it is a fibration in the $\hf$-model structure on $\Fun(\C,\D)$ and the diagram
\vspace{-2mm}
\begin{center}
\begin{tikzpicture}[node distance=2.2cm, auto, scale=1]

\node (A) {$F$};
\node (B) [right of=A] {$\widehat P_nF$};
\node (C) [node distance=1.6cm, below of=A] {$G$};
\node (D) [node distance=1.6cm, below of=B] {$\widehat P_nG$};

\draw[->] (A) to node {$\widehat p_F$} (B);
\draw[->] (A) to node [swap] {$\alpha$} (C);
\draw[->] (C) to node [swap] {$\widehat p_G$} (D);
\draw[->] (B) to node {$\widehat P_n\alpha$} (D);

\end{tikzpicture}
\end{center}
is a homotopy pullback square in the $\hf$-model structure.  A  functor $F$  in $\Fun(\C,\D)$ is fibrant in the $n$-excisive model structure if and only if it is weakly equivalent in the $\hf$-model structure to an $n$-excisive functor and is fibrant in the $\hf$-model structure.
\end{prop}

\begin{proof}
    The first statement is an immediate consequence of Proposition \ref{Qfibrations}.  For the second part, one can show that $\widehat p_\ast \colon \ast\rightarrow \widehat P_n(\ast)$, where $\ast$ denotes the terminal object in $\Fun(\C,\D)$, is a weak equivalence in the projective model structure on $\Fun(\C,\D)$ using Lemma \ref{hf pres terminal}, the facts that $T_n$ preserves weak equivalences and that $T_n(\ast)$ is weakly equivalent to $\ast$, and Convention \ref{HomotopyConventions}(\ref{preseqcolims}). Since $(-)^{\hf}$ satisfies (A1), the map $\widehat p_\ast \colon \ast\rightarrow \widehat P_n(\ast)$ is also an $\hf$-equivalence.  By Proposition \ref{QpresWE} it suffices to prove that $\widehat p_F$ is an $\hf$-equivalence if and only if $F$ is $\hf$-equivalent to an $n$-excisive functor. 

    Consider the diagram 
    \vspace{-2mm}
    \begin{center}
\begin{tikzpicture}[node distance=2.2cm, auto, scale=1]

\node (A) {$F$};
\node (B) [right of=A] {$\widehat P_nF$};
\node (C) [node distance=1.6cm, below of=A] {$G$};
\node (D) [node distance=1.6cm, below of=B] {$\widehat P_nG.$};

\draw[->] (A) to node {$\widehat p_F$} (B);
\draw[->] (A) to node [swap] {$\beta$} (C);
\draw[->] (C) to node [swap] {$\widehat p_G$} (D);
\draw[->] (B) to node {$\widehat P_n\beta$} (D);
\end{tikzpicture}
\end{center}
If $\beta$ is an $hf$-weak equivalence, then $\widehat P_n\beta$ is an $\hf$-weak equivalences since $\widehat P_n$ preserves $\hf$-weak equivalences. Moreover, as noted in the proof of (A2) for Theorem \ref{t:nexcisive}, if $G$ is $n$-excisive, the maps in the colimit defining $\widehat P_nG$ are all weak equivalences and as a result, $\iota_{G^{\hf}} \colon G^{\hf} \rightarrow \widehat P_nG$ is a levelwise weak equivalence. Since $(-)^{\hf}$ satisfies $(A2)$ and preserves levelwise weak equivalences, it follows that the composite $\widehat p_G$ is an $\hf$-equivalence.  By Proposition \ref{properwe}, $\widehat p_F$ is a weak equivalence, establishing one implication.  The converse follows immediately from the fact that $\widehat P_nF$ is $n$-excisive.
\end{proof}

\section{The degree $n$ model structure} \label{Degreenmodel}

In this section, we turn our attention to the discrete functor calculus, with the goal of showing that one can equip $\Fun(\C,\D)$ with a degree $n$ model structure via the degree $n$ approximation $\Gamma_n$.  As was the case with the homotopy functor and $n$-excisive model structures, our categories $\C$ and $\D$ must satisfy some conditions, but they are much less complicated in this case. 

\begin{convention} \label{degreenconvention}
In addition to Convention \ref{conventiondiscrete}, we assume that $\D$ is also left proper. 
\end{convention}

Recall from Definition \ref{degreen} that for a functor $F$ and an object $A$ in $\C$, 
\[ \Gamma_n F(A):= \hocofiber \left(|\perp_{n+1}^{*+1} F(A)| \rightarrow F(A)\right), \] 
where $|-|$ denotes the homotopy colimit over $\Deltaop$, sometimes referred to as the \emph{fat realization}.   We establish the existence of the degree $n$ model structure on $\Fun(\C,\D)$ by proving that $\Gamma_n$ satisfies the conditions of Theorem \ref{BousfieldQ}, for which we make use of specific models for the homotopy limits and colimits used to define $\Gamma_n$.  

For the homotopy limit, we again use the model defined in \cite[18.1.8]{hirschhorn}.  As noted in Section \ref{s:discretefc}, we can assume that $\perp_{n+1}$ is a comonad.  To guarantee that $\perp_{n+1}^k$ preserves weak equivalences, we replace $F$ with its functorial fibrant replacement $\mathbb{F}(F)$, which exists by Theorem \ref{SimpFibRep}.  That is, we set 
\[ \bot_{n+1}^{k}F:=\bot_{n+1}^{k}\mathbb{F}(F)  \]
for any functor $F$.  Then by \cite[18.5.2, 18.5.3]{hirschhorn} we know that $\bot_{n+1}^{k}$ preserves weak equivalences in $\Fun(\C,\D)$.  

For the model of $\Gamma_n$ that we use, we also require a good model for homotopy colimits, for which we use the one described by Hirschhorn \cite[18.1.2]{hirschhorn}.  To ensure that $|-|$ preserves weak equivalences, we precompose the homotopy colimit with the simplicial functorial cofibrant replacement functor guaranteed by \cite[24.2]{shulman}.  Again, \cite[18.5.3]{hirschhorn} guarantees that this functor preserves levelwise weak equivalences of diagrams.

Finally, to guarantee that the homotopy cofiber preserves weak equivalences, and that there is a natural transformation from the identity functor on $\Fun(\C,\D)$ to $\Gamma_n$, we use the following model for the homotopy cofiber.  For a map $F\rightarrow G$, we set $\hocofiber(F\rightarrow G)$ equal to the pushout of the diagram $E(\star_\D)\hookleftarrow F \rightarrow G$ where $\star_\mathcal D$ is the constant functor on the zero object in $\D$, and $F\rightarrow E(\star_\mathcal D)\rightarrow \star_\mathcal D$ is a functorial factorization of $F\rightarrow \star_\mathcal D$ as a cofibration followed by an acyclic fibration.  That this construction is homotopy invariant follows from \cite[13.5.3, 13.5.4]{hirschhorn} and the fact that $\D$ is left proper.  

When using the construction defined above, there is a natural map from   $G$ to  $\hocofiber(F\rightarrow G)$; an application of this fact to the augmentation $|\bot_{n+1}^{*+1} F| \rightarrow \mathbb{F}(F)$ yields a natural map $\mathbb F(F) \rightarrow \Gamma_n(F)$, and precomposing with the natural weak equivalence $\theta_F \colon F\rightarrow \mathbb{F}(F)$ yields a natural transformation $F \rightarrow \Gamma_nF$ that is natural in $F$.  Hence, we have a natural transformation $\gamma \colon \id_{\Fun(\C,\D)}\rightarrow \Gamma_n$. 

Recall that as part of Convention \ref{conventiondiscrete}, we are assuming $\D$ is stable.  As noted in Section \ref{s:fcbackground}, the proof that the functor $\Gamma_nF$ is a degree $n$ functor in \cite[5.4, 5.6.1]{bjm} makes use of the fact that $\D$ is stable to ensure that $cr_{n+1}$ commutes with the homotopy cofiber and colimit used to define $\Gamma_nF$. We thus obtain
\begin{align*}
cr_{n+1}\Gamma_nF&=cr_{n+1}\hocofiber\left (|\bot_{n+1}^{\ast+1}F|\rightarrow F\right)\\
&\simeq \hocofiber\left(|cr_{n+1}(\bot_{n+1}^{\ast+1}F)|\rightarrow cr_{n+1}F\right)\\
&\simeq \star_{\D}.
\end{align*}
The last equivalence is a consequence of \cite[5.5]{bjm} which uses an extra degeneracy argument to prove that $|cr_{n+1}(\bot_{n+1}^{\ast+1}F)|\rightarrow cr_{n+1}F$ is a weak equivalence.

We now state the main result of this section.

\begin{theorem} \label{degreenmodel}
Under the assumptions of Convention \ref{degreenconvention}, there exists a degree $n$ model structure on the category of functors $\Fun(\C,\D)$ with weak equivalences given by $\Gamma_n$-equivalences and cofibrations the same as in the projective model structure. 
\end{theorem}

\begin{proof}
We prove this result via an application of Corollary \ref{simplifiedA3}, setting $\M=\Fun(\C,\D)$ and $Q=\Gamma_n \colon \Fun(\C,\D) \rightarrow \Fun(\C,\D)$.  As described above,  the functor $\Gamma_n \colon \Fun(\C,\D)\rightarrow \Fun(\C,\D)$ is constructed from $F$ via homotopy limits, homotopy colimits, and homotopy cofibers that preserve weak equivalences.  Hence, axiom (A1) of Theorem \ref{BousfieldQ} is satisfied.  

To prove that axiom (A2) holds, we assume that $F$ is fibrant and  use Corollary \ref{fibrantA2}. By \cite[18.5.2, 18.5.3]{hirschhorn}, the results of applying $\bot_{n+1}^{k}$ to $F$ and $\mathbb{F}(F)$ are weakly equivalent, so we can work directly with $F$ instead of its functorial fibrant replacement. 

To prove that $\gamma_{\Gamma_nF}$ is a weak equivalence, we note that for any degree $n$ functor $G$, 
\[ \bot_{n+1}G=\Delta^\ast cr_{n+1}G\simeq \star_{\D} \]
where $\Delta^\ast$ denotes precomposition with the diagonal functor.  Hence, the simplicial object $\bot_{n+1}^{\ast+1}G$ is levelwise weakly equivalent to the constant simplicial object on $\star_{\D}$ and   $|\bot_{n+1}^{\ast+1}G|\simeq \star_{\D}$. Since $\Gamma_nG$ is the homotopy cofiber of $|\bot_{n+1}^{\ast+1}G|\rightarrow G$ and $\D$ is left proper, the dual of Proposition \ref{properwe} guarantees that  $\gamma_G$ is a weak equivalence.  Setting $G=\Gamma_nF$, we see that $\gamma_{\Gamma_nF}$ is a weak equivalence.  

To prove that $\Gamma_n\gamma_F$ is a weak equivalence, it suffices to show that $\Gamma_n|\bot_{n+1}^{\ast+1}F|\simeq \star_{\D}.$  As a comonad, $\bot_{n+1}$ comes equipped with a natural transformation, the comultiplication $\bot_{n+1}\Rightarrow \bot_{n+1}\bot_{n+1}$, which can be used 
to construct weak equivalences between $|\bot_{n+1}^k\bot_{n+1}^{\ast +1}F|$ and $\bot_{n+1}^kF$ for $k\geq 1$, as in \cite[5.5]{bjm}.  Since $\D$ is stable, $\bot_{n+1}$, as a finite homotopy limit, commutes with $|-|$, so we have 
\[ |\bot_{n+1}^{\ast+1}|\bot_{n+1}^{\ast +1}F||\simeq ||\bot_{n+1}^{\ast+1}\bot_{n+1}^{\ast+1} F||\simeq |\bot_{n+1}^{\ast +1}F|.\]
It follows that  $\Gamma_n|\bot_{n+1}^{\ast +1}F|\simeq \star_{\D}$, which implies that $\Gamma_n\gamma_F \colon \Gamma_nF\rightarrow \Gamma_n\Gamma_nF$ is a weak equivalence.

It remains to check that $\Gamma_n$ satisfies axiom (A3'), i.e., that homotopy pullback squares are preserved by $\Gamma_n$.  Suppose that
\vspace{-3mm}
\begin{center}
\begin{tikzpicture}[node distance=1.7cm, auto, scale=1]

\node (A) {$F$};
\node (B) [right of=A] {$G$};
\node (C) [node distance=1.4cm, below of=A] {$H$};
\node (D) [node distance=1.4cm, below of=B] {$K$};

\draw[->] (A) to node {} (B);
\draw[->] (A) to node {} (C);
\draw[->] (B) to node {} (D);
\draw[->] (C) to node {} (D);

\end{tikzpicture}
\end{center}
is a homotopy pullback square in $\Fun(\C,\D)$ and recall that homotopy pullback and homotopy pushout squares agree in a stable model category such as $\D.$   Since $\Gamma_n$ is constructed via homotopy limits and colimits, homotopy limits preserve homotopy pullbacks, and homotopy colimits preserve homotopy pushouts, it follows that applying $\Gamma_n$ to this diagram yields a homotopy pullback square.
\end{proof}

We conclude with the following consequence of  Proposition \ref{QpresWE}.  Its proof is similar to that of Proposition \ref{fibrantexcisive} and hence omitted.  

\begin{prop} \label{fibrantdegreen}
A morphism $\alpha \colon F\rightarrow G$ in $\Fun(\C,\D)$ is a fibration in the degree $n$ model structure if and only if it is a fibration in the projective model structure on $\Fun(\C,\D)$ and the diagram
\vspace{-3mm}
\begin{center}
\begin{tikzpicture}[node distance=2cm, auto, scale=1]

\node (A) {$F$};
\node (B) [right of=A] {$\Gamma_nF$};
\node (C) [node distance=1.6cm, below of=A] {$G$};
\node (D) [node distance=1.6cm, below of=B] {$\Gamma_nG$};

\draw[->] (A) to node {$\gamma_F$} (B);
\draw[->] (A) to node [swap] {$\alpha$} (C);
\draw[->] (C) to node [swap] {$\gamma_G$} (D);
\draw[->] (B) to node {$\Gamma_n\alpha$} (D);

\end{tikzpicture}
\end{center}
is a homotopy pullback square.  The object $F$ in $\Fun(\C,\D)$ is fibrant in the degree $n$ model structure if and only if it is degree $n$ and is fibrant in the projective model structure.
\end{prop}

\section{Cofibrant generation for Bousfield $Q$-model structures} \label{s:cofibrant}

In this section, we establish conditions under which the localizations produced by an endofunctor $Q$ on $\Fun(\C,\D)$ via Theorem \ref{BousfieldQ} are cofibrantly generated, expanding on the more specific examples of Biedermann and R\"ondigs in \cite{br}. A core element of those examples is the strategic creation of additional generating acyclic cofibrations for the model structure $\Fun(\C,\D)_Q$.  Since our goal is to generalize those examples, we make the following definition to capture the key features of the collections of maps that they use.  We assume throughout that $\C$ and $\D$ are as described in Convention \ref{convention}. 

\begin{definition} \label{testdefn} 
Let $\Fun(\C,\D)$ be equipped with a right proper model structure and let $Q$ be an endofunctor of $\Fun(\C,\D)$ satisfying the conditions of Theorem \ref{BousfieldQ}, including the existence of a natural transformation $\eta \colon \id \Rightarrow Q$.  A \emph{collection of test morphisms} for $\eta$ is a collection $T(Q)$ of morphisms in $\Fun(\C, \s)$ such that, for each fibration $F \rightarrow G$ in $\Fun(\C,\D)$, the diagram
\begin{center}
\begin{tikzpicture}[node distance=2cm, auto, scale=1]

\node (A) {$F$};
\node (B) [right of=A] {$QF$};
\node (C) [node distance=1.5cm, below of=A] {$G$};
\node (D) [node distance=1.5cm, below of=B] {$QG$};

\draw[->] (A) to node {} (B);
\draw[->] (A) to node [swap] {} (C);
\draw[->] (C) to node [swap] {} (D);
\draw[->] (B) to node {} (D);

\end{tikzpicture}
\end{center}
is a homotopy pullback in $\Fun(\C,\D)$ if and only if the diagram
\begin{center}
\begin{tikzpicture}[node distance=2cm, auto, scale=1]

\node (A) {$F^Y$};
\node (B) [right of=A] {$F^X$};
\node (C) [node distance=1.5cm, below of=A] {$G^Y$};
\node (D) [node distance=1.5cm, below of=B] {$G^X$};

\draw[->] (A) to node {} (B);
\draw[->] (A) to node [swap] {} (C);
\draw[->] (C) to node [swap] {} (D);
\draw[->] (B) to node {} (D);

\end{tikzpicture}
\end{center}
is a homotopy pullback in $\D$ for every $X\rightarrow Y$ in $T(Q)$.
\end{definition}

We sometimes omit the endofunctor $Q$ from the notation and simply refer to the collection of test morphisms as $T$. In all the examples in this paper, the functor $Y$ is a representable functor $R^A$, where $A$ is an object of $\C$.  

\begin{example} \label{hfTestMorphisms}
 Consider the $\hf$-model structure on $\Fun(\C,\D)$ from Theorem \ref{hfmodel}, where $\C$ and $\D$ additionally satisfy the conditions of Convention \ref{HomotopyConventions}.  Since it is designed to emphasize weak equivalence-preserving functors, we claim that the collection of morphisms of representable functors $\{R^B \rightarrow R^A\}$, where $A \rightarrow B$ ranges over all weak equivalences of $\C$, is a collection of test morphisms for $\eta \colon \id \Rightarrow (-)^{\hf}$.  Indeed, our definition of a collection of test morphisms is essentially a distillation of the key properties that Biedermann and R\"ondigs use in \cite{br} to show that the $\hf$-model structure is cofibrantly generated. That these morphisms satisfy Definition \ref{testdefn} was proved by Biedermann and R\"ondigs in \cite[4.15]{br} using the fact that $F^{R^A} \cong F(A)$ by the Yoneda Lemma \ref{Yoneda}. 
\end{example}

We revisit this example at the end of this section and give additional examples in Sections \ref{s:nexccof} and \ref{s:degreencof}, where we consider the $n$-excisive and degree $n$ model structures, respectively.

We can now state and prove our main result.  

\begin{theorem} \label{finalcofib} 
Suppose that $\C$ and $\D$ satisfy the conditions of Convention \ref{nexcisiveconv}, and  that $\Fun(\C,\D)$ is a cofibrantly generated right proper model structure on $\Fun(\C,\D)$ in which all fibrations are also fibrations under the projective model structure.  Let $Q \colon \Fun(\C,\D) \rightarrow \Fun(\C,\D)$ be a Bousfield endofunctor that has a collection $T(Q)$ of test morphisms for the natural transformation $\eta \colon \id \Rightarrow Q$. Then the model structure $\Fun(\C,\D)_Q$ is cofibrantly generated. 
\end{theorem}

\begin{proof} 
Let $I$ and $J$ denote sets of generating cofibrations and acyclic cofibrations, respectively, for the model structure on $\Fun(\C,\D)$.  We need to identify sets of generating cofibrations and acyclic cofibrations for the model structure $\Fun(\C,\D)_Q$.   Since the cofibrations are unchanged by the $Q$-localization, we can simply use the set $I$ as a set of the generating cofibrations for $\Fun(\C,\D)_Q$. 

Assume that $F\rightarrow G$ is a fibration in $\Fun(\C,\D)$. To identify a set of generating acyclic cofibrations for $\Fun(\C,\D)_Q$, we use  the $Q$-fibration condition of Proposition \ref{Qfibrations}, namely that $F \rightarrow G$ is a $Q$-fibration if and only if it is a fibration in $\Fun(\C,\D)$, and the diagram
\vspace{-3mm}
\begin{equation}\label{5-7square}
\begin{tikzpicture}[baseline=(current bounding box.center), node distance=2cm, auto, scale=1]

\node (A) {$F$};
\node (B) [right of=A] {$QF$};
\node (C) [node distance=1.5cm, below of=A] {$G$};
\node (D) [node distance=1.5cm, below of=B] {$QG$};

\draw[->] (A) to node {$\eta_F$} (B);
\draw[->] (A) to node [swap] {} (C);
\draw[->] (C) to node [swap] {$\eta_G$} (D);
\draw[->] (B) to node {} (D);

\end{tikzpicture}
\end{equation}
is a homotopy pullback square in $\Fun(\C,\D)$.  We know that a map is a fibration in $\Fun(\C,\D)$ if and only if it has the right lifting property with respect to the set $J$.  Thus, it suffices to identify a set of maps $J_Q$ such that the above diagram is a homotopy pullback square if and only if $F \rightarrow G$ has the right lifting property with respect to the maps in $J_Q$; we can then take $J \cup J_Q$ as the set of generating acyclic cofibrations for $\Fun(\C,\D)_Q$.

By the definition of test morphism, \eqref{5-7square} is a homotopy pullback square if and only if 
\vspace{-3mm}
\begin{equation}\label{cotensorsquare}
\begin{tikzpicture}[baseline=(current bounding box.center), node distance=2cm, auto, scale=1]

\node (A) {$F^Y$};
\node (B) [right of=A] {$F^X$};
\node (C) [node distance=1.5cm, below of=A] {$G^Y$};
\node (D) [node distance=1.5cm, below of=B] {$G^X$};

\draw[->] (A) to node {} (B);
\draw[->] (A) to node [swap] {} (C);
\draw[->] (C) to node [swap] {} (D);
\draw[->] (B) to node {} (D);

\end{tikzpicture}
\end{equation}
is a homotopy pullback square for each $\alpha \colon X\rightarrow Y$ in $T$.  

However, for each object $X$ in $\Fun(\C,\s)$ we can apply Proposition \ref{cylinder} to the morphism from the initial object to $X$ to obtain a simplicial homotopy equivalence 
\begin{equation} \label{cofrep}
\widehat X \rightarrow X
\end{equation}
where $\widehat X$ is cofibrant.  Since simplicial functors preserve simplicial homotopy equivalences \cite[9.6.10]{hirschhorn}, we have that $F^X\rightarrow F^{\widehat X}$ and $G^X\rightarrow G^{\widehat X}$ are weak equivalences.  For each test morphism $\alpha \colon X\rightarrow Y$, we obtain a diagram in which the right-hand square is a homotopy pullback square by Proposition \ref{properwe}:
\vspace{-1mm}
\begin{equation}
\label{cofibrantsquare}
\begin{tikzpicture}[baseline=(current bounding box.center), node distance=2cm, auto, scale=1]

\node (A) {$F^Y$};
\node (B) [right of=A] {$F^X$};
\node (X) [right of=B] {$F^{\widehat X}$};
\node (C) [node distance=1.5cm, below of=A] {$G^Y$};
\node (D) [node distance=1.5cm, below of=B] {$G^X$};
\node (Y) [node distance=1.5cm, below of=X] {$G^{\widehat X}$.};

\draw[->] (A) to node {} (B);
\draw[->] (A) to node [swap] {} (C);
\draw[->] (C) to node [swap] {} (D);
\draw[->] (B) to node {} (D);
\draw[->] (B) to node {} (X);
\draw[->] (X) to node {} (Y);
\draw[->] (D) to node {} (Y);
\end{tikzpicture}
\end{equation}
Applying Proposition \ref{pullbackcomposite} to this diagram, we see that \eqref{cotensorsquare} is a homotopy pullback if and only if the outer square in \eqref{cofibrantsquare} is a homotopy pullback square.

Again by Proposition \ref{cylinder} and the fact that simplicial functors preserve simplicial homotopy equivalences, we know that the composite $\widehat X\rightarrow X \xrightarrow{\alpha} Y$ can be factored as a cofibration $\zeta_\alpha$ followed by a simplicial homotopy equivalence $\widehat{X} \xrightarrow{\zeta_\alpha} \Cyl(\alpha) \xrightarrow{\simeq} Y$, and we obtain weak equivalences of evaluated cotensors $F^Y\rightarrow F^{\Cyl(\alpha)}$ and $G^Y\rightarrow G^{\Cyl(\alpha)}$.  Then, by Proposition \ref{pullbackcomposite}, the outer square in \eqref{cofibrantsquare} is a homotopy pullback square if and only if
\vspace{-3mm}
\begin{equation} \label{cylindersquare}
\begin{tikzpicture}[baseline=(current bounding box.center), node distance=2.2cm, auto, scale=1]

\node (A) {$F^{\Cyl(\alpha)}$};
\node (B) [right of=A] {$F^{\widehat X}$};
\node (C) [node distance=1.5cm, below of=A] {$G^{\Cyl(\alpha)}$};
\node (D) [node distance=1.5cm, below of=B] {$G^{\widehat X}$};

\draw[->] (A) to node {} (B);
\draw[->] (A) to node [swap] {} (C);
\draw[->] (C) to node [swap] {} (D);
\draw[->] (B) to node {} (D);

\end{tikzpicture}
\end{equation}
is a homotopy pullback square.

Since $F\rightarrow G$ is a projective fibration by assumption, by Corollary \ref{cotensorwefib}  we know that $F^{\widehat X}\rightarrow G^{\widehat X}$ is a fibration, so \eqref{cylindersquare} is a homotopy pullback if and only if 
\begin{equation} \label{pullbackmap}
F^{\Cyl(\alpha)}\rightarrow F^{\widehat X}\times _{G^{\widehat X}} G^{\Cyl(\alpha)}
\end{equation}
is a weak equivalence.  Note that \eqref{pullbackmap} is a fibration by Lemma \ref{pullbackgen}.  So we can show that it is weak equivalence by showing that a lift exists in every commutative diagram of the form 
\vspace{-2mm}
\begin{equation}\label{liftsquare}
\begin{tikzpicture}[baseline=(current bounding box.center), node distance=3cm, auto, scale=1]

\node (A) {$C$};
\node (B) [right of=A] {$F^{\Cyl(\alpha)}$};
\node (C) [node distance=1.5cm, below of=A] {$D$};
\node (D) [node distance=1.5cm, below of=B] {$F^{\widehat X}\times_{G^{\widehat X}}G^{\Cyl(\alpha)}$,};

\draw[->] (A) to node {} (B);
\draw[->] (A) to node [swap] {} (C);
\draw[->] (C) to node [swap] {} (D);
\draw[->] (B) to node {} (D);
\draw[dashed, ->] (C) to node {} (B);

\end{tikzpicture}
\end{equation}
where $C\rightarrow D$ is in the set $I_\D$ of generating cofibrations for $\D$.  By the first adjunction of Proposition \ref{FunAdjunctions}, a lift exists in \eqref{liftsquare} if and only if a lift exists in 
\vspace{-2mm}
\begin{equation*} \begin{tikzpicture}[node distance=3.5cm, auto, scale=1]

\node (A) {$D\otimes \widehat X\coprod_{C\otimes \widehat X}C\otimes \Cyl(\alpha)$};
\node (B) [right of=A] {$F$};
\node (C) [node distance=1.5cm, below of=A] {$D\otimes \Cyl(\alpha)$};
\node (D) [node distance=1.5cm, below of=B] {$G$.};

\draw[->] (A) to node {} (B);
\draw[->] (A) to node [swap] {} (C);
\draw[->] (C) to node [swap] {} (D);
\draw[->] (B) to node {} (D);
\draw[dashed, ->] (C) to node {} (B);

\end{tikzpicture} \end{equation*}
\vspace{-1mm}
Hence, the left-hand vertical maps can be taken as the set $J_Q$. That is, 
\[ J_Q=\{ i\Box \zeta_{\alpha} \colon C\otimes \Cyl(\alpha)\amalg_{C\otimes \widehat X}D\otimes \widehat X\rightarrow D\otimes \Cyl(\alpha) \mid (i \colon C\rightarrow D) \in I_{\D}, \alpha\in T \}. \]

The preceding argument shows that the maps in $J_Q$ have the left lifting property with respect to all $Q$-fibrations, so that they are indeed acyclic cofibrations in the model structure induced by $Q$.  To complete the proof, it remains to show that this set of maps permits the small object argument.  We need to show that, given a transfinite composition $\colim_n F_n$ of pushouts along acyclic cofibrations, any map 
\[ \langle f, g \rangle \colon C\otimes \Cyl(\alpha)\amalg_{C\otimes \widehat X}D\otimes \widehat X \rightarrow \colim_n F_n \] 
factors through some $F_k$. By the universal property of pushouts, the data of such a map is equivalent to a commutative square 
\vspace{-2mm}
\[ \begin{tikzpicture}[node distance=3.5cm, auto, scale=1]

\node (A) {$C\otimes \widehat X$};
\node (B) [right of=A] {$D\otimes \widehat X$};
\node (C) [node distance=1.5cm, below of=A] {$C\otimes \Cyl(\alpha)$};
\node (D) [node distance=1.5cm, below of=B] {$\colim_n F_n$,};

\draw[->] (A) to node {$i \otimes 1$} (B);
\draw[->] (A) to node [swap] {$1 \otimes \zeta_{\alpha}$} (C);
\draw[->] (C) to node [swap] {$f$} (D);
\draw[->] (B) to node {$g$} (D);

\end{tikzpicture} \]
\vspace{-1mm}
and a factorization of such a map through a functor $F_k$ is equivalent to a commutative diagram  
\vspace{-2mm}
\[ \begin{tikzpicture}[node distance=4cm, auto, scale=1]

\node (A) {$C\otimes \widehat X$};
\node (B) [right of=A] {$D\otimes \widehat X$};
\node (C) [node distance=1.8cm, below of=A] {$C\otimes \Cyl(\alpha)$};
\node (D) [node distance=1.8cm, below of=B] {$\colim_n F_n$.};

\node (E) [node distance=2cm, right of=A] {};
\node (F) [node distance=9mm, below of=E] {$F_k$};

\draw[->] (A) to node {$i \otimes 1$} (B);
\draw[->] (A) to node [swap] {$1 \otimes \zeta_{\alpha}$} (C);
\draw[->] (C) to node [swap] {$f$} (D);
\draw[->] (B) to node {$g$} (D);

\draw[->] (C) to node {$f_k$} (F);
\draw[->] (B) to node {$g_k$} (F);
\draw[->] (F) to node {} (D);

\end{tikzpicture} \]
Using the first adjunction of Proposition \ref{FunAdjunctions}, this commutative diagram is in turn equivalent to
\vspace{-2mm}
\[ \begin{tikzpicture}[node distance=4.4cm, auto, scale=1]

\node (A) {$C$};
\node (B) [right of=A] {$(\colim_n F_n)^{\Cyl(\alpha)}$};
\node (C) [node distance=3.2cm, below of=A] {$D$};
\node (D) [node distance=3.2cm, below of=B] {$(\colim_n F_n)^{\widehat X}$.};

\node (E) [node distance=2.2cm, right of=A] {};
\node (F) [node distance=8mm, below of=E] {$F_k^{\Cyl(\alpha)}$};
\node (G) [node distance=2.4cm, below of=E] {$F_k^{\widehat X}$};

\draw[->] (A) to node {$\widehat{f}$} (B);
\draw[->] (A) to node [swap] {$i$} (C);
\draw[->] (C) to node [swap] {$\widehat{g}$} (D);
\draw[->] (B) to node {$1^{\zeta_{\alpha}}$} (D);

\draw[->] (C) to node {$\widehat{g}_k$} (G);
\draw[->] (G) to node {} (D);
\draw[->] (A) to node [swap] {$\widehat{f}_k$} (F);
\draw[->] (F) to node {} (B);
\draw[->] (F) to node {} (G);

\end{tikzpicture} \]
The fact that both $C$ and $D$ permit the small object argument, by Convention \ref{nexcisiveconv}, guarantees the existence of maps $\widehat{f}_k$ and $\widehat{g}_k$ such that the diagram above commutes. 
\end{proof}

As a first example of this localized cofibrant generation, the next theorem follows immediately from Theorem \ref{finalcofib} and Example \ref{hfTestMorphisms}; an alternate proof is given in \cite[4.14]{br}.

\begin{theorem} 
Assuming Convention \ref{nexcisiveconv}, the model category $\Fun(\C,\D)_{\hf}$ of Theorem \ref{hfmodel} has the structure of a cofibrantly generated model category.
\end{theorem}

When $\Fun(\C,\D)$ has the projective model structure,  Theorem \ref{finalcofib} implies that for any Bousfield endofunctor $Q$ that has a collection of test morphisms,  the model structure $\Fun(\C,\D)_Q$ of Theorem \ref{BousfieldQ} is cofibrantly generated.  Moreover, since the fibrations of $\Fun(\C,\D)_Q$ must also be projective fibrations by Proposition \ref{Qfibrations}, we can use Theorem \ref{finalcofib} to conclude that the localization of $\Fun(\C,\D)$ obtained by applying a sequence of Bousfield endofunctors satisfying the appropriate test morphism conditions is cofibrantly generated.  We provide an example of this in the next section by building an $n$-excisive model structure on $\Fun(\C,\D)$ from the $\hf$-model structure.  To do so, we make use of the next result.  

\begin{prop} \label{Pfibration}
Consider the category $\Fun(\C,\D)$ with the projective model structure, as well as a localized model structure $\Fun(\C,\D)_P$ induced by a Bousfield endofunctor $P$ of $\Fun(\C,\D)$.  Suppose that $Q$ is an endofunctor of $\Fun(\C,\D)_P$ that preserves $P$-fibrations.  Then for any $P$-fibration $F \rightarrow G$, the diagram
\begin{equation} \label{xqxsquare}
\begin{tikzpicture}[baseline=(current bounding box.center), node distance=1.8cm, auto, scale=1]

\node (A) {$F$};
\node (B) [right of=A] {$QF$};
\node (C) [node distance=1.4cm, below of=A] {$G$};
\node (D) [node distance=1.4cm, below of=B] {$QG$};

\draw[->] (A) to node {} (B);
\draw[->] (A) to node [swap] {} (C);
\draw[->] (C) to node [swap] {} (D);
\draw[->] (B) to node {} (D);

\end{tikzpicture}
\end{equation}
is a homotopy pullback square in the $P$-model structure if and only if it is a levelwise homotopy pullback.
\end{prop}

\begin{proof} 
By assumption, $QF \rightarrow QG$ is a $P$-fibration, and hence, a projective fibration. So the pullback of $G \rightarrow QG \leftarrow QF$ is a homotopy pullback in both the projective and $P$-model structures.

If \eqref{xqxsquare} is a levelwise homotopy pullback square, then the map \(F \rightarrow QF \times_{QG} G \) 
is a levelwise weak equivalence.  Since $P$ satisfies axiom (A1) of Theorem \ref{BousfieldQ}, this map is also a $P$-equivalence, and  \eqref{xqxsquare} is a homotopy pullback square in $\Fun(\C,\D)_P$.

Conversely, suppose that the diagram \eqref{xqxsquare} is a homotopy pullback square in the $P$-model structure, so that the map  
\(PF \rightarrow P(QF \times_{QG} G )\) 
is a levelwise weak equivalence.  Since $P$ satisfies axiom (A3') of Corollary \ref{simplifiedA3}, the map $P(QF \times_{QG} G) \to PQF \times^h_{PQG} PG$
is a levelwise weak equivalence, and the right-hand square in the diagram
\begin{equation} \label{Phmtpypullback} \begin{tikzpicture}[baseline=(current bounding box.center), node distance=2cm, auto, scale=1]

\node (A) {$F$};
\node (B) [right of=A] {$PF$};
\node (X) [right of=B] {$PQF$};
\node (C) [node distance=1.4cm, below of=A] {$G$};
\node (D) [node distance=1.4cm, below of=B] {$PG$};
\node (Y) [node distance=1.4cm, below of=X] {$PQG$,};

\draw[->] (A) to node {} (B);
\draw[->] (A) to node [swap] {} (C);
\draw[->] (C) to node [swap] {} (D);
\draw[->] (B) to node {} (D);
\draw[->] (B) to node {} (X);
\draw[->] (X) to node {} (Y);
\draw[->] (D) to node {} (Y);

\end{tikzpicture}
\end{equation}
is a homotopy pullback square in $\Fun(\C,\D)$.  By Proposition  \ref{Qfibrations}, the left-hand square is a levelwise homotopy pullback as well, from which we can conclude by Proposition \ref{pullbackcomposite} that the outer square is also a levelwise homotopy pullback square.

However, this outer square can  be obtained similarly as a composite diagram with middle vertical map the $P$-fibration $QF \rightarrow QG$, so that the left-hand square is the square \eqref{xqxsquare}.  We can conclude that this square is a levelwise homotopy pullback square via another application of Proposition \ref{Qfibrations} and Proposition \ref{pullbackcomposite}, completing the proof.
\end{proof}

\section{Cofibrant generation and the $n$-excisive model structure} \label{s:nexccof}

In this section, we apply Theorem \ref{finalcofib} and Proposition \ref{Pfibration} to show that the $n$-excisive model structure of Theorem \ref{t:nexcisive} is cofibrantly generated.  As in Section \ref{n-excisive model}, we assume Convention \ref{nexcisiveconv}. 

We first define a candidate set of test morphisms for the $n$-excisive model structure.

\begin{definition} \label{defnexctest}
For an object $A$ in $\C$, let $\tau_A$ be the morphism 
\[ \tau_A \colon \underset{U\subseteq \P_0(\mathbf{n+1})}{\hocolim}R^{A\ast U}\longrightarrow R^{A} \]
in $\Fun(\C,\s)$ induced by the inclusions $\varnothing\hookrightarrow U$, where $\ast$ denotes the fiberwise join as in Definition \ref{Pndefn}.  We denote by $T(\widehat P_n)$ the collection $\{\tau_A\}$ of these morphisms as $A$ ranges over all objects of $\C$.
\end{definition}

\begin{prop} \label{n-exctest} 
For the hf-model structure $\Fun(\C,\D)_{\hf}$, the collection $T(\widehat P_n)$ is a collection of test morphisms for the natural transformation $\eta \colon \id \Rightarrow P_n$.
\end{prop}

\begin{proof}  
Let $F\rightarrow G$ be an $\hf$-fibration.  By Proposition \ref{Pnpreshf}, we know that $\widehat P_n$ preserves $\hf$-fibrations.  Hence, by Proposition \ref{Pfibration}, it suffices to show that
\vspace{-2mm}
\begin{equation} \label{Pnpullback}
\begin{tikzpicture}[baseline=(current bounding box.center), node distance=2cm, auto, scale=1]

\node (A) {$F$};
\node (B) [right of=A] {$\widehat{P}_nF$};
\node (C) [node distance=1.4cm, below of=A] {$G$};
\node (D) [node distance=1.4cm, below of=B] {$\widehat{P}_nG$};

\draw[->] (A) to node {} (B);
\draw[->] (A) to node [swap] {} (C);
\draw[->] (C) to node [swap] {} (D);
\draw[->] (B) to node {} (D);

\end{tikzpicture}
\end{equation} 
is a levelwise homotopy pullback square if and only if
\vspace{-2mm}
\begin{equation} \label{cotensorpullback}
\begin{tikzpicture}[baseline=(current bounding box.center), node distance=3.5cm, auto, scale=1]

\node (A) {$F^{R^A}$};
\node (B) [right of=A] {$F^{\hocolim_{U\in \P_0(\mathbf{n+1})}R^{A\ast U}}$};
\node (C) [node distance=1.4cm, below of=A] {$G^{R^A}$};
\node (D) [node distance=1.4cm, below of=B] {$G^{\hocolim_{U\in \P_0(\mathbf{n+1})}R^{A\ast U}}$};

\draw[->] (A) to node {} (B);
\draw[->] (A) to node [swap] {} (C);
\draw[->] (C) to node [swap] {} (D);
\draw[->] (B) to node {} (D);

\end{tikzpicture}
\end{equation}
is a homotopy pullback square in $\D$ for all objects $A$ in $\C$.  Combining the isomorphism of Proposition \ref{evcotensorholim}, Lemma \ref{Yoneda}, and the definition of $T_n$ (Definition \ref{Pndefn}), we see that \eqref{cotensorpullback} is a homotopy pullback square for all objects $A$ in $\C$ if and only if 
\vspace{-2mm}
\begin{equation}\label{Tnpullback}
\begin{tikzpicture}[baseline=(current bounding box.center), node distance=2cm, auto, scale=1]

\node (A) {$F$};
\node (B) [right of=A] {$T_nF$};
\node (C) [node distance=1.4cm, below of=A] {$G$};
\node (D) [node distance=1.4cm, below of=B] {$T_nG$};

\draw[->] (A) to node {} (B);
\draw[->] (A) to node [swap] {} (C);
\draw[->] (C) to node [swap] {} (D);
\draw[->] (B) to node {} (D);

\end{tikzpicture}
\end{equation}
is a levelwise homotopy pullback square.  So it suffices to show that \eqref{Pnpullback} is a levelwise homotopy pullback if and only if \eqref{Tnpullback} is.

Suppose \eqref{Pnpullback} is a levelwise homotopy pullback and consider the commutative cube 
\vspace{-2mm}
\begin{equation}\label{TnPncube}
\begin{tikzpicture}[baseline=(current bounding box.center), node distance=2.5cm, auto, scale=1]

\node (A) {$T_nF$};
\node (B) [right of=A] {$T_n\widehat{P}_nF$};
\node (C) [node distance=1.8cm, below of=A] {$T_nG$};
\node (D) [node distance=1.8cm, below of=B] {$T_n\widehat{P}_nG$.};

\node (X) [node distance=9mm, above of=A] {};

\node (E) [node distance=1.5cm, left of=X] {$F$};
\node (F) [right of=E] {$\widehat{P}_nF$};
\node (G) [node distance=1.8cm,below of=E] {$G$};
\node (H) [node distance=1.8cm,below of=F] {$\widehat{P}_nG$};

\draw[->] (A) to node {} (B);
\draw[->] (A) to node {} (C);
\draw[->] (C) to node {} (D);
\draw[->] (B) to node {} (D);

\draw[->] (E) to node {} (A);
\draw[->] (F) to node {} (B);
\draw[->] (E) to node {} (F);
\draw[->, dashed] (F) to node {} (H);
\draw[->, dashed] (H) to node {} (D);

\draw[->, dashed] (G) to node {} (H);
\draw[->] (G) to node {} (C);
\draw[->] (E) to node {} (G);

\end{tikzpicture}
\end{equation}
The front and back faces are homotopy pullbacks by assumption and the fact that $T_n$ preserves homotopy pullbacks, respectively.  If we consider its back and right faces, 
we see that the right face is a homotopy pullback because $\widehat{P_n}F$ is $n$-excisive, and hence the composite of the back and right faces is a homotopy pullback by Proposition \ref{pullbackcomposite}.  As a result, we can conclude that the composite of the left and front faces is a homotopy pullback. Since  the front face is a homotopy pullback, we can apply Proposition \ref{pullbackcomposite} again to see that the left face, which is precisely \eqref{Tnpullback}, is as well.

Conversely, suppose that \eqref{Tnpullback} is a homotopy pullback square.  Consider the commutative diagram that defines $\widehat P_n$ and the natural transformations in \eqref{Pnpullback}:
\vspace{-2mm}
\begin{equation}\label{colimit}
\begin{tikzpicture}[baseline=(current bounding box.center), node distance=2cm, auto, scale=1]

\node (A) {$F$};
\node (B) [right of=A] {$F^{\hf}$};
\node (X) [right of=B] {$T_nF^{\hf}$};
\node (U) [right of=X] {$\hdots$};
\node (N) [node distance=3.3cm, right of=U] {$\colim_{k} T_n^kF^{\hf} = \widehat{P}_nF$};

\node (C) [node distance=1.5cm, below of=A] {$G$};
\node (D) [node distance=1.5cm, below of=B] {$G^{\hf}$};
\node (Y) [node distance=1.5cm, below of=X] {$T_nG^{\hf}$};
\node (V) [node distance=1.5cm, below of=U] {$\hdots$};
\node (M) [node distance=1.5cm, below of=N] {$\colim_{k} T_n^kG^{\hf} = \widehat{P}_nG$.};

\draw[->] (A) to node {} (B);
\draw[->] (B) to node {} (X);
\draw[->] (X) to node {} (U);
\draw[->] (U) to node {} (N);

\draw[->] (A) to node [swap] {} (C);
\draw[->] (C) to node [swap] {} (D);
\draw[->] (X) to node [swap] {} (Y);
\draw[dashed,->] (N) to node [swap] {} (M);

\draw[->] (B) to node {} (D);
\draw[->] (D) to node {} (Y);
\draw[->] (Y) to node {} (V);
\draw[->] (V) to node {} (M);
\end{tikzpicture}
\end{equation}
It suffices to show that for each $k\geq 0$, the square 
\vspace{-1mm}
\begin{equation}\label{hfpullback}
\begin{tikzpicture}[baseline=(current bounding box.center), node distance=2.2cm, auto, scale=1]

\node (A) {$F$};
\node (B) [right of=A] {$T^k_nF^{\hf}$};
\node (C) [node distance=1.4cm, below of=A] {$G$};
\node (D) [node distance=1.4cm, below of=B] {$T^k_nG^{\hf}$,};

\draw[->] (A) to node {} (B);
\draw[->] (A) to node [swap] {} (C);
\draw[->] (C) to node [swap] {} (D);
\draw[->] (B) to node {} (D);

\end{tikzpicture}
\end{equation}
whose horizontal maps are given by composites of horizontal maps in \eqref{colimit}, is a homotopy pullback square.  In the case that $k=0$, this  follows from Proposition \ref{Qfibrations}.  

Assuming \eqref{hfpullback} is a homotopy pullback for some $k\geq 0$, we see that the right-hand square of the commutative diagram
\vspace{-3mm}
\begin{equation}\label{k+1square}
\begin{tikzpicture}[node distance=2.2cm, auto, scale=1]

\node (A) {$F$};
\node (B) [right of=A] {$T_nF$};
\node (X) [right of=B] {$T^{k+1}_nF^{\hf}$};
\node (C) [node distance=1.4cm, below of=A] {$G$};
\node (D) [node distance=1.4cm, below of=B] {$T_nG$};
\node (Y) [node distance=1.4cm, below of=X] {$T_n^{k+1}G^{\hf}$};

\draw[->] (A) to node {} (B);
\draw[->] (A) to node [swap] {} (C);
\draw[->] (C) to node [swap] {} (D);
\draw[->] (B) to node {} (D);
\draw[->] (B) to node {} (X);
\draw[->] (X) to node {} (Y);
\draw[->] (D) to node {} (Y);

\end{tikzpicture}
\end{equation}
is a homotopy pullback since it is obtained by applying the functor $T_n$, which preserves homotopy pullbacks, to \eqref{hfpullback}.  The outer square is then a homotopy pullback square by Proposition \ref{pullbackcomposite} since the left-hand square is \eqref{Tnpullback}.  The commutative cube 
\vspace{-2mm}
\begin{equation}\label{cube}
\begin{tikzpicture}[baseline=(current bounding box.center), node distance=2.5cm, auto, scale=1]

\node (A) {$T_nF$};
\node (B) [right of=A] {$T_n^{k+1}F^{\hf}$};
\node (C) [node distance=1.8cm, below of=A] {$T_nG$};
\node (D) [node distance=1.8cm, below of=B] {$T_n^{k+1}G^{\hf}$};

\node (X) [node distance=9mm, above of=A] {};

\node (E) [node distance=1.5cm, left of=X] {$F$};
\node (F) [right of=E] {$F^{\hf}$};
\node (G) [node distance=1.8cm,below of=E] {$G$};
\node (H) [node distance=1.8cm,below of=F] {$G^{\hf}$};

\draw[->] (A) to node {} (B);
\draw[->] (A) to node {} (C);
\draw[->] (C) to node {} (D);
\draw[->] (B) to node {} (D);

\draw[->] (E) to node {} (A);
\draw[->] (F) to node {} (B);
\draw[->] (E) to node {} (F);
\draw[->, dashed] (F) to node {} (H);
\draw[->, dashed] (H) to node {} (D);

\draw[->, dashed] (G) to node {} (H);
\draw[->] (G) to node {} (C);
\draw[->] (E) to node {} (G);

\end{tikzpicture}
\end{equation}
shows us that the outer square in \eqref{k+1square} is the same square as  \eqref{hfpullback} when $k$ is replaced by $k+1$, completing the proof by induction.
\end{proof}

We can now conclude the  main result of this section from Theorem \ref{finalcofib}.

\begin{theorem}\label{PnCof}
    The $n$-excisive model structure on $\Fun(\C,\D)$ from Theorem \ref{t:nexcisive} is cofibrantly generated.
\end{theorem}

\begin{remark}
Theorem \ref{PnCof} is also proved as part of \cite[5.8]{br}.  Our proof is essentially a reorganization and generalization of theirs via Theorem \ref{finalcofib}, but our approaches differ in the stage at which the representable functors of Definition \ref{defnexctest} are introduced.  In \cite{br}, the authors incorporate them into their definition of $P_n$ by replacing $T_nF$ with an evaluated cotensor
\[ T_n^{R}F(A):=F^{A_n} \]
where
\vspace{-2mm}
\begin{center}
\begin{tikzpicture}[node distance=3.8cm, auto, scale=1]

\node (A) {$A_n$};
\node (B) [right of=A] {$\hocolim_{U\in \P_0(\mathbf{n+1}) }R^{A\ast U}$};

\draw[->] (A) to node {$\simeq$} (B);

\end{tikzpicture}
\end{center}
is a cofibrant replacement for the homotopy colimit of the representable functors $R^{A\ast U}$.  We proved in \cite[7.4]{proc} that $T_n$ and $T_n^R$ agree up to weak equivalence.  We have chosen to define $\widehat P_n$ without using an evaluated cotensor to highlight the fact that this approach is not needed to establish the existence of the $n$-excisive model structure.  It does play a significant role in establishing cofibrant generation, since, in the proof of Theorem \ref{finalcofib}, being able to replace \eqref{5-7square} with the evaluated cotensor square \eqref{cotensorsquare} provides the means by which we can identify a set of generating acyclic cofibrations, but now the specific evaluated cotensor approach to $\widehat P_n$ only appears concretely in our verification of our set of test morphisms in Proposition \ref{n-exctest}.  
\end{remark}

\section{Cofibrant generation and discrete functor calculus}\label{s:degreencof}

We now revisit the degree $n$ model structure of Section \ref{Degreenmodel} and use Theorem \ref{finalcofib} to show that it is cofibrantly generated when $\D$ is.  As in Section \ref{Degreenmodel} we assume Convention \ref{degreenconvention}.

Recall from Definition \ref{degreen} and Section \ref{Degreenmodel} that the functor $\Gamma_nF$ is defined in terms of a comonad $\bot_{n+1}$ that acts on the category $\Fun(\C,\D).$ More explicitly, it is the homotopy cofiber given by  
\[ \Gamma_n F:=\hocofiber(\vert \bot_{n+1}^{\ast+1}\mathbb{F}(F)\vert \rightarrow \mathbb{F}(F)) \]
where $\vert \bot_{n+1}^{\ast+1}F\vert$ is the fat realization of the standard simplicial object associated to the comonad $\bot_{n+1}$ acting on $F$ and $\mathbb{F}(F)$ is a functorial fibrant replacement of $F$. 

\begin{theorem} \label{degncofgen}
The degree $n$ model structure on $\Fun(\C,\D)$ from Theorem \ref{degreenmodel} is cofibrantly generated.  
\end{theorem}

We want to prove this theorem by an application of Theorem \ref{finalcofib}, for which we need to define a collection of test morphisms for $\Gamma_n$.  Recall from \eqref{e:Ai} that for an object $A$ in $\mathcal{C}$, subset $U$ of $\mathbf{n+1}$, and element $i \in \mathbf{n+1}$, we defined 
\[ A_i(U):= \begin{cases}
A & i\notin U, \\
\ast_\mathcal{C} & i\in U.
\end{cases} \]
Using this definition, we define
\[ \sqcup(A,U):= \underset{i\in \mathbf{n+1}}\coprod A_i(U), \]
and note that for $U\subseteq V$ there is a natural map 
\[ \iota_{U,V} \colon \sqcup(A,U) \longrightarrow \sqcup(A,V)\]
induced by the unique morphism $A \rightarrow \ast_\C$ on the components indexed by $i \in V \setminus U$. 

\begin{definition} \label{degreentestdefn} 
For an object $A$ in $\C$, let $\rho_A$ be the morphism 
\[ \rho_A \colon \underset{U\subseteq \P_0(\mathbf{n+1})}{\hocolim}R^{\sqcup(A,U)} \rightarrow R^{\sqcup(A,\varnothing)} \]
in $\Fun(\C,\s)$ induced by the morphisms $\iota_{{\varnothing},U}$.  We denote by $T(\Gamma_n)$ the collection $\{\rho_A\}$ of these morphisms as $A$ ranges over all objects of $\C$.
\end{definition}

\begin{lemma} \label{degreentestlemma} 
For the projective model structure on $\Fun(\C,\D)$, the collection $T(\Gamma_n)$ of Definition \ref{degreentestdefn} is a collection of test morphisms for the natural transformation  $\gamma \colon id \Rightarrow \Gamma_n$. 
\end{lemma}

To prove this lemma, we use the next two results.

\begin{lemma} \label{techlemma}
Let $\C$ be a subcategory of a model category that is closed under finite limits.  If $\X$ is an $n$-cube in $\C$, then 
\[ \ifiber(\X) \simeq \hofiber \left(\X(\varnothing) \rightarrow \holim_{\mathcal P_0(\mathbf n)} \X(U) \right). \]
\end{lemma}

This lemma was proved in the context of spaces in \cite[3.4.3]{mv} for 2-cubes and \cite[5.5.4]{mv} for general $n$-cubes; the same line of argument holds in this more general setting. 

The proof of the following lemma is a straightforward exercise, using Proposition \ref{pullbackcomposite} and its dual.  

\begin{lemma} \label{fibpullback}
For a commutative square
\vspace{-2mm}
\begin{equation}\label{abcd}
\begin{tikzpicture}[baseline=(current bounding box.center), node distance=1.8cm, auto, scale=1]

\node (A) {$A$};
\node (B) [right of=A] {$B$};
\node (C) [node distance=1.4cm, below of=A] {$C$};
\node (D) [node distance=1.4cm, below of=B] {$D$,};

\draw[->] (A) to node {$\alpha$} (B);
\draw[->] (A) to node {} (C);
\draw[->] (C) to node [swap] {$\gamma$} (D);
\draw[->] (B) to node {} (D);

\end{tikzpicture}
\end{equation}
in a pointed right proper model category $\D$, the induced map of homotopy fibers 
\[ \hofiber(\alpha)\rightarrow \hofiber(\gamma) \]
is a weak equivalence if  the square \eqref{abcd} is a homotopy pullback square.  If $\D$ is stable and proper, the converse is true as well.
\end{lemma}

\begin{proof}[Proof of Lemma \ref{degreentestlemma}] 
Let $F\rightarrow G$ be a fibration in $\Fun(\C,\D)$.  By an argument similar to the one used to start the proof of Proposition \ref{n-exctest},  it suffices to show that 
\vspace{-2mm}
\begin{equation} \label{Gammansquare}
\begin{tikzpicture}[baseline=(current bounding box.center), node distance=2cm, auto, scale=1]

\node (A) {$F$};
\node (B) [right of=A] {$\Gamma_nF$};
\node (C) [node distance=1.4cm, below of=A] {$G$};
\node (D) [node distance=1.4cm, below of=B] {$\Gamma_n G$};

\draw[->] (A) to node {$\gamma_F$} (B);
\draw[->] (A) to node {} (C);
\draw[->] (C) to node [swap] {$\gamma_G$} (D);
\draw[->] (B) to node {} (D);

\end{tikzpicture}
\end{equation}
is a homotopy pullback if and only if 
the diagram 
\vspace{-1mm}
\begin{equation}\label{Yonedasquare}
\begin{tikzpicture}[baseline=(current bounding box.center), node distance=4cm, auto, scale=1]

\node (A) {$F(\sqcup(A,\varnothing))$};
\node (B) [right of=A] {$\underset{U\subseteq\P_0(\mathbf{n+1})}{\holim}F(\sqcup(A,U))$};
\node (C) [node distance=1.5cm, below of=A] {$F(\sqcup(A,\varnothing))$};
\node (D) [node distance=1.5cm, below of=B] {$\underset{U\subseteq\P_0(\mathbf{n+1})}{\holim}G(\sqcup(A,U))$};

\draw[->] (A) to node {} (B);
\draw[->] (A) to node {} (C);
\draw[->] (C) to node [swap] {} (D);
\draw[->] (B) to node {} (D);

\end{tikzpicture}
\end{equation}
is a homotopy pullback for all objects $A$ in $\C.$

We can write \eqref{Gammansquare} as the composite
\vspace{-2mm}
\[ \begin{tikzpicture}[node distance=2.2cm, auto]

\node (A) {$F$};
\node (B) [right of=A] {${\mathbb F}(F)$};
\node (E) [right of=B] {$\Gamma_n F$};
\node (C) [node distance=1.4cm, below of=A] {$G$};
\node (D) [right of=C] {${\mathbb F}(G)$};
\node (F) [right of=D] {$\Gamma_n G$};

\draw[->] (A) to node {$\simeq$} (B);
\draw[->] (B) to node {} (E);
\draw[->] (C) to node  {$\simeq$} (D);
\draw[->] (D) to node {} (F);
\draw[->] (A) to node {} (C);
\draw[->] (B) to node {} (D);
\draw[->] (E) to node {} (F);

\end{tikzpicture} \] 
by  definition of the natural transformation $\gamma:\id_{\Fun(\C,\D)}\rightarrow\Gamma_n$.  Using the dual of Proposition \ref{pullbackcomposite} and the fact that $\D$ is stable, we see that the right-hand square is a homotopy pullback if and only if the outer square is a homotopy pullback.  Similarly, \eqref{Yonedasquare} is a homotopy pullback if and only if the corresponding square with fibrant replacements of $F$ and $G$ on the left is a homotopy pullback square.  Hence, we can restrict to the case where we use fibrant replacements of $F$ and $G$ and  for simplicity, we suppress the fibrant replacement notation $\mathbb F$ for the remainder of the proof.  

By Lemma \ref{techlemma}, the homotopy fibers of the top and bottom horizontal arrows in \eqref{Yonedasquare} are $\bot_{n+1}F(A)$ and $\bot_{n+1}G(A)$, respectively.  Then by Lemma \ref{fibpullback}, the diagram \eqref{Yonedasquare} is a homotopy pullback if and only if the induced  map of homotopy fibers $\bot_{n+1}F(A)\rightarrow \bot_{n+1}G(A)$ is a weak equivalence.

If \eqref{Yonedasquare} is a homotopy pullback for all objects $A$ in $\C$, it follows that 
$\vert \bot_{n+1}^{\ast+1}F\vert\rightarrow \vert \bot_{n+1}^{\ast+1}G\vert$ is a weak equivalence in $\Fun(\C,\D)$. Consider the diagram
\vspace{-2mm}
\begin{center}
\begin{tikzpicture}[node distance=2.5cm, auto]

\node (A) {$\vert\bot_{n+1}^{\ast +1}F\vert$};
\node (B) [right of=A] {$F$};
\node (E) [right of=B] {$\Gamma_n F$};
\node (C) [node distance=1.5cm, below of=A] {$\vert\bot_{n+1}^{\ast +1}G\vert$};
\node (D) [right of=C] {$G$};
\node (F) [right of=D] {$\Gamma_n G$,};

\draw[->] (A) to node {} (B);
\draw[->] (B) to node {} (E);
\draw[->] (C) to node {} (D);
\draw[->] (D) to node {} (F);
\draw[->] (A) to node {} (C);
\draw[->] (B) to node {} (D);
\draw[->] (E) to node {} (F);

\end{tikzpicture}
\end{center}
where the top and bottom rows are the homotopy cofiber sequences defining $\Gamma_nF$ and $\Gamma_nG$. respectively. Since $\D$ is stable, the rows are also homotopy fiber sequences, and the right-hand square, which is exactly \eqref{Gammansquare}, is a homotopy pullback by Lemma \ref{fibpullback}.

Conversely, if \eqref{Gammansquare} is a homotopy pullback, then the square
\vspace{-2mm}
\begin{center}
\begin{tikzpicture}[node distance=3cm, auto]

\node (A) {$\bot_{n+1}F$};
\node (E) [right of=A] {$\bot_{n+1}\Gamma_n F$};
\node (C) [node distance=1.5cm, below of=A] {$\bot_{n+1}G$};
\node (F) [right of=C] {$\bot_{n+1}\Gamma_n G$.};

\draw[->] (A) to node {} (E);
\draw[->] (A) to node {} (C);
\draw[->] (C) to node {} (F);
\draw[->] (E) to node {} (F);

\end{tikzpicture}
\end{center}
is also a homotopy pullback as $\bot_{n+1}$ preserves homotopy pullbacks.  As noted in the proof of Theorem \ref{degreenmodel}, $\bot_{n+1}\Gamma_nF\simeq \ast\simeq \bot_{n+1}\Gamma_nG$, so $\bot_{n+1}F \rightarrow \bot_{n+1}G$ is a weak equivalence by Proposition \ref{properwe}, and  \eqref{Yonedasquare} is a homotopy pullback for all objects $A$ in $\C$.  
\end{proof}

\begin{proof}[Proof of Theorem \ref{degncofgen}]
The proof of Theorem \ref{degreenmodel} establishes that $\Gamma_n$ is a Bousfield endofunctor on $\Fun(\C,\D)$ with the projective model structure.  Proposition \ref{Qfibrations} guarantees that the fibrations in $\Fun(\C,\D)_{\Gamma_n}$ are also projective fibrations.  Hence, we can apply Theorem \ref{finalcofib}, using the collection of test morphisms for $\Gamma_n$ that we established in Lemma \ref{degreentestlemma}.
\end{proof}

\end{document}